\documentclass{amsart}
\usepackage{graphicx}
\usepackage{amsmath}
\usepackage{epstopdf} 
\usepackage{amsfonts}
\usepackage{multirow}
\usepackage{color}
\usepackage{xspace}
\usepackage{array}
\usepackage{tikz}
\usetikzlibrary{positioning}
\usetikzlibrary{arrows}
\usepackage{pgfplots}
\usepackage{algorithm}
\usepackage{algorithmic}
\newcommand{\rd}{\mathrm{d}}
\newcommand{\RR}{\mathbb{R}}
\newcommand{\Sb}{\mathbb{S}}
\newcommand{\hd}{h_{\mathrm{d}}}

\newcommand{\ud}[2]{u^{\left( #1 \right)}_{ #2 }}
\newcommand{\tud}[2]{\tilde{u}^{\left( #1 \right)}_{ #2 }}
\newcommand{\vd}[2]{v^{\left( #1 \right)}_{ #2 }}
\newcommand{\taud}[2]{\tau^{\left( #1 \right)}_{ #2 }}
\newcommand{\ed}[2]{e^{\left( #1 \right)}_{ #2 }}
\newcommand{\rhod}[2]{\rho^{\left( #1 \right)}_{ #2 }}
\newcommand{\fd}{\delta^+}
\newcommand{\bd}{\delta^-}
\newcommand{\cd}[1][1]{\delta^{\langle #1 \rangle}}
\newcommand{\ifd}{\sigma^+}
\newcommand{\ibd}{\sigma^-}
\newcommand{\icd}[1][1]{\sigma^{\langle #1 \rangle}}

\newcommand{\fa}{\mu^+}
\newcommand{\ba}{\mu^-}

\newcommand{\Range}{\mathop{\mathrm{range}}\nolimits}
\newcommand{\Null}{\mathop{\mathrm{null}}\nolimits}
\newcommand{\Car}{\mathop{\mathrm{car}}\nolimits}
\newcommand{\im}{\mathrm{i}}
\newcommand{\pinv}[1]{{#1}^{\dagger}}

\newtheorem{theorem}{Theorem}[section]
\newtheorem{lemma}[theorem]{Lemma}
\newtheorem{corollary}[theorem]{Corollary}
\newtheorem{proposition}[theorem]{Proposition}

\theoremstyle{remark}
\newtheorem{remark}[theorem]{Remark}
\begin{document}
	

\title[Stability and convergence of conservative scheme for modified Hunter--Saxton equation]{Stability and convergence of\\ a conservative finite difference scheme\\ for the modified Hunter--Saxton equation}


\author{Shun Sato}
\address{Graduate School of Information Science and Technology, The University of Tokyo, Bunkyo-ku, Tokyo, Japan}
\email{shun\_sato@mist.i.u-tokyo.ac.jp}
\thanks{The first author was supported in part by JSPS Research Fellowship for Young Scientists.}

\date{February, 2018}


\maketitle

\begin{abstract}
The modified Hunter--Saxton equation models the propagation of short capillary-gravity waves.
As it involves a mixed derivative, its initial value problem on the periodic domain is much more complicated than the standard evolutionary equations.
Although its local well-posedness has recently been proved, the behavior of its solution is yet to be investigated.
In this paper, to develop a reliable numerical method for this problem,
we derive a conservative finite difference scheme.
Then, we rigorously prove not only its stability in the sense of the uniform norm but also its uniform convergence to sufficiently smooth exact solutions.
Discrete conservation laws are used to overcome the difficulty due to the mixed derivative.
\end{abstract}

\section{Introduction}
\label{sec_intro}

This paper focuses on numerical integration of the initial value problem for the modified Hunter--Saxton (mHS) equation~\cite{FMN2007}
\begin{equation}\label{eq_mHS}
	\begin{cases}
		u_{tx} + \frac{1}{2} \left(u^2 \right)_{xx} = 2 \omega u + \frac{1}{2} u_x^2 \qquad & (t \in (0,T), x \in \Sb ), \\
		u (0,x) = u_0 (x) & (x \in \Sb)
	\end{cases}
\end{equation}
on the periodic domain $ \Sb := \RR / L \mathbb{Z} $, where $ \omega \in \RR $ is a nonzero constant
and the subscripts $t$ and $x$ denote partial derivatives.

The mHS equation models the propagation of short capillary-gravity waves.
The case $ \omega = 0$ corresponds to the Hunter--Saxton (HS) equation~\cite{HS1991},
which was originally derived to investigate the nonlinear instability in the director field of a nematic liquid crystal.
The initial value problem for the HS equation on the whole real line $\RR $ or the half line $ [0,+ \infty) $ (the Dirichlet boundary condition is imposed at $ x = 0$)
has been studied intensively in both partial differential equation (PDE) theory and numerical analysis~(see \cite{HKR2007} and the references therein).
In addition, the HS equation on the periodic domain has attracted considerable attention. In this case, as it is underdetermined owing to the presence of the mixed derivative $ u_{tx} $,
some techniques for choosing an appropriate solution have been studied
(see \cite{Y2004,MCFM2017} and the references therein for details).

On the other hand, Li and Yin~\cite{LY2017_1} recently showed that the initial value problem for the periodic mHS equation~\eqref{eq_mHS} is locally well-posed.
In this case, although the mixed derivative $ u_{tx} $ is involved,
the implicit constraint
\begin{equation}\label{eq_ic}
	\mathcal{F} (u(t)) = 0, \qquad \mathcal{F} (v) := \int_{\Sb} \left( 2 \omega v + \frac{1}{2} v_x^2 \right) \rd x
\end{equation}
provides us with the reformulation into the standard form $ u_t = ... $,
which enables us to use Kato's theory~\cite{K1975} (see \cite{SM2017+2} for details on this type of reformulation).
Moreover, Li and Yin~\cite{LY2017_1} showed the $L^{\infty}$ bound of the solution by using a conservation law,
\begin{equation}\label{eq_norm}
	\mathcal{H} (u(t) ) = \mathcal{H} (u_0) , \qquad \mathcal{H} (v) := \frac{1}{2} \int_{\Sb}  v_x^2 \rd x,
\end{equation}
and the implicit constraint~\eqref{eq_ic}
(see Proposition~\ref{prop_snb}).
Furthermore, they reported blow up phenomena with respect to $ \| u_x \|_{\infty} $ (see Section~\ref{subsec_bup} for details on blow up phenomena).

\begin{remark}\label{rem_LY}
	Li and Yin~\cite{LY2017_1} actually dealt with the initial value problem for
	\begin{equation}\label{eq_mHS_LY}
	U_{TX} = U + 2U U_{XX} +U_X^2.
	\end{equation}
	on the domain $ \RR / \mathbb{Z} $.
	However, the solution $U$ of \eqref{eq_mHS_LY} is related to the solution $u$ of~\eqref{eq_mHS} via the scaling
	\[ U(T,X) = - \frac{1}{4 \omega L^2 } u \left( \frac{T}{2 \omega L } , L X \right).  \]
	Therefore, in this paper, we refer to their results by appropriately extending them to our case~\eqref{eq_mHS}.
\end{remark}

However, numerical investigation of the mHS equation has not been conducted sufficiently thus far.
The only related numerical study is that by Miyatake, Cohen, Furihata, and Matsuo~\cite{MCFM2017},
who considered another problem related to the mHS equation.
Specifically,
as their study was conducted before the local well-posedness result was obtained~\cite{LY2017_1},
they considered how to obtain the traveling wave solution (discovered by Lenells~\cite{L2009}) of the underdetermined form
\begin{equation}\label{eq_mHS2}
	u_{txx} + \frac{1}{2} \left(u^2 \right)_{xxx} = 2 \omega u_x + \frac{1}{2} \left( u_x^2 \right)_x
\end{equation}
of the mHS equation.
They reported that their method facilitates long-term reproduction of the traveling wave solution smoothly and effectively
owing to the discrete preservation of $ \mathcal{H} $.
However, their method deals with obtaining the traveling wave solution; it does not deal with our problem~\eqref{eq_mHS}, and there is no other mathematical analysis, such as unique existence and convergence.

Such mathematical analysis is particularly challenging
when the target equation involves the mixed derivative $ u_{tx} $.
At present, convergence analysis for such equations can be found only in the work of Coclite, Ridder, and Risebro~\cite{CRR2017}.
They proved the convergence of the numerical solution of some finite difference schemes to the unique entropy solution of the reduced Ostrovsky equation~\cite{H1990}
\begin{equation}\label{eq_rO}
	u_{tx} + \frac{1}{2} \left( u^2 \right)_{xx} = \gamma u,
\end{equation}
which models the propagation of water waves on a very shallow rotating fluid ($ \gamma $ is a nonzero parameter).
Note that their result can be applied to equations in the general form $ u_{tx} + ( f(u))_{xx} = \gamma u $, including the short pulse equation~\cite{SW2004} ($ f : \RR \to \RR $ is assumed to be $C^2$).

As the mHS equation~\eqref{eq_mHS} resembles the reduced Ostrovsky equation~\eqref{eq_rO},
extension of the convergence result may seem straightforward.
However, this is not the case, because the simplicity of the right-hand side of~\eqref{eq_rO} is an essential requirement.
Actually, as discussed in~\cite{SM2017+2}, the reduced Ostrovsky equation (and the more general case of $ u_{tx} + (f(u))_{xx} = \gamma u $) is the simplest one in the class of evolutionary equations involving the mixed derivative.
Specifically, owing to the presence of the linear implicit constraint $ \int_{\Sb} u(t,x) \rd x = 0 $,
the reduced Ostrovsky equation can be rewritten in the standard form of evolutionary equations,
\[ u_t + \frac{1}{2} \left(u^2 \right)_x = \int_0^x u(t,y) \rd y - \frac{1}{L} \int_{\Sb} \int_0^z u(t,y) \rd y \rd z, \]
which resembles scalar conservation laws.
As the right-hand side is not too problematic,
Coclite, Ridder, and Risebro~\cite{CRR2017} used the knowledge of numerical methods for scalar conservation laws.

As mentioned above, the mHS equation can be rewritten in the standard form of evolutionary equations,
but the right-hand side involves quadratic nonlinearity of the spatial derivative.
This makes mathematical analysis of the numerical method a challenging task.
In this paper, we derive a stable finite difference scheme for the mHS equation,
whose numerical solution uniquely exists and converges to the sufficiently smooth exact solution.

In view of the known blow up results of the mHS equation,
we expect that its solutions tend to develop sharp fronts.
Therefore, to safely conduct numerical experiments,
some special treatment is indispensable.
Toward this end, in this paper,
we try to derive a numerical scheme inheriting the $L^{\infty} $ bound shown by Li and Yin~\cite{LY2017_1}.
As the $L^{\infty} $ bound comes from invariants $ \mathcal{F} $ and $ \mathcal{H} $,
we construct a numerical method preserving their discrete counterparts.
Although the preservation of multiple invariants is known to be computationally expensive in the literature on structure-preserving methods (see, e.g., \cite{MQR1998}),
we can efficiently achieve the discrete preservation of $ \mathcal{F} $ and $ \mathcal{H} $ in this case
owing to some special property of $ \mathcal{H} $ (see Section~\ref{sec_des}).
Moreover, we reveal the relation between our scheme and the scheme of Miyatake, Cohen, Furihata, and Matsuo~\cite{MCFM2017},
which shows that the mathematical analysis presented in this paper can be extended to their numerical method (see Remark~\ref{rem_equiv}).

Then, we prove the unique existence of the numerical solution in Section~\ref{subsec_ue}.
To achieve a tighter sufficient condition for unique existence,
we adopt some discrete reformulation into a form suitable for the contraction mapping theorem.
Through this approach, we achieve the mild sufficient condition $ \Delta t = O (\Delta x) $ (Lemmas~\ref{lem_ball} and \ref{lem_contraction}).

Subsequently, we show the $L^{\infty} $ global error estimate (Corollary~\ref{cor_conv}) of our numerical scheme in Section~\ref{subsec_conv}.
To overcome the difficulty due to the mixed derivative~$ u_{tx} $,
we combine the standard local truncation error estimate (Lemma~\ref{lem_lte}) and the average error estimate (Lemma~\ref{lem_ae}),
which comes from the discrete conservation laws of $ \mathcal{H} $ and $ \mathcal{F} $.

It should be emphasized that, except for the simplest case, such as the reduced Ostrovsky equation,
our contribution is the first rigorous convergence analysis of the numerical method for evolutionary equations having the mixed derivative $ u_{tx} $.
In particular, the present result deals with the nonlinear implicit constraint~\eqref{eq_ic}, while
the reduced Ostrovsky equation has the linear implicit constraint~$ \int_{\Sb} u(t,x) \rd x = 0 $ (see~\cite{SM2017+2} for difficulties that arise when the implicit constraint is nonlinear).
Therefore, we believe that this achievement signifies the possibility of rigorous justification of existing and future numerical methods for such equations (see, e.g., \cite{YMS2010,MYM2012,OS2012,FSM2016} for existing results without convergence analysis).

In addition, we conduct two numerical experiments.
One is to numerically confirm the theoretical convergence rate,
and the other is to observe new blow up phenomena.
The latter result may suggest an interesting direction for future studies on the analytical aspects of the mHS equation.

The remainder of this paper is organized as follows.
Section~\ref{sec_pre} presents some preliminaries, including several discrete inequalities.
Section~\ref{sec_conti} reviews the $L^{\infty} $ bound of the exact solution and associated invariants.
As mentioned earlier, the main result is stated in Sections~\ref{sec_des} and~\ref{sec_anal}.
Section~\ref{sec_ne} is devoted to numerical experiments for confirming the convergence rate and observing new blow up phenomena.
Finally, Section~\ref{sec_cr} concludes the paper.

\section{Preliminaries}
\label{sec_pre}

Let us introduce the discrete symbol $ \ud{m}{k} \approx u(m\Delta t, k \Delta x) \ ( m = 0,\dots ,M ; k \in \mathbb{Z} )$ and
impose the discrete periodic boundary condition $ \ud{m}{k+K} = \ud{m}{k} \ ( m = 0,\dots, M; k \in \mathbb{Z})$,
where $ K \in \mathbb{Z} $ is a constant satisfying $ L = K \Delta x $.
Here, $ \Delta t $ and $ \Delta x $ are temporal and spatial mesh sizes, respectively (in this paper, we deal with only the uniform mesh for simplicity).
In view of the discrete periodicity, we often use the notation $ \ud{m}{} := \left( \ud{m}{1} ,\dots , \ud{m}{K} \right)^{\top} $.
The contents of this section are based on the work of Furihata and Matsuo~\cite{FM2011}.

We define the discrete Lebesgue space $ L^p_K (\Sb) $ ($ 1 \le p \le \infty $) as the pair $ \left( \RR^K , \| \cdot \|_p \right) $, where the norm is defined as
\begin{equation}\label{def_norm_d}
	\| v \|_p := \left( \sum_{k=1}^K \left| v_k \right|^p \Delta x \right)^{\frac{1}{p}} \ (1 \le p < \infty), \qquad
	\| v \|_{\infty} := \max_{ k \in \mathbb{Z}} \left| v_k \right|.
\end{equation}

For $ L^2_K (\Sb)$, the associated inner product $ \langle \cdot , \cdot \rangle $ is defined as
\[ \langle v , w \rangle = \sum_{k=1}^K v_k w_k \Delta x. \]
For simplicity, we abbreviate $ \| \cdot \|_2 $ by omitting the subscript $2$ hereafter.

It should be noted that, in Section~\ref{sec_conti}, we use the notation $ \| \cdot \|_p $ for the standard (continuous) $p$-norms.
Although this is an abuse of notation, we use it because no confusion occurs.

\subsection{Difference and average operators}

Let us introduce the forward, backward, and central difference operators:
\begin{align*}
	\fd_x v_k &:= \frac{v_{k+1} - v_k}{\Delta x}, &
	\bd_x v_k &:= \frac{v_{k} - v_{k-1}}{\Delta x}, &
	\cd_x v_k &:= \frac{v_{k+1} - v_{k-1}}{2\Delta x}.
\end{align*}
Note that $ \fd_x $ is the operator such that $ \fd_x : \RR^K \to \RR^K $ and
$ ( \fd_x v )_k $ is abbreviated as $ \fd_x v_k $.
Then, the second-order central difference operator $ \cd[2]_x $ is defined as
\[ \cd[2]_x v_k = \frac{v_{k+1} - 2 v_k + v_{k-1}}{(\Delta x)^2}. \]
Moreover, the forward and backward average operators are defined as follows:
\begin{align*}
	\fa_x v_k &:= \frac{v_{k+1} + v_k}{2}, &
	\ba_x v_k &:= \frac{v_{k} +v_{k-1}}{2}.
\end{align*}
Similarly, we use the temporal forward difference and average operators defined as
\begin{align*}
	\fd_t \ud{m}{k} &= \frac{\ud{m+1}{k} - \ud{m}{k}}{\Delta t},&
	\fa_t \ud{m}{k} &= \frac{\ud{m+1}{k} + \ud{m}{k}}{2}.
\end{align*}

Then, several basic properties hold as follows (their proofs are straightforward):

\begin{lemma}
	The following properties hold:
	\begin{enumerate}
		\item All the above-mentioned operators commute with each other;
		\item $ \cd_x = \fd_x \ba_x = \bd_x \fa_x $;
		\item $ \cd[2]_x = \fd_x \bd_x $;
		\item $ \fd_x (v_k w_k) = ( \fa_x v_k ) ( \fd_x w_k) + ( \fd_x v_k ) ( \fa_x w_k ) $;
		\item Summation by parts formulae corresponding to the skew-symmetry of $ \partial_x $:
		\begin{align*}
			\langle \fd_x v , w \rangle &= - \langle v , \bd_x w \rangle,&
			\langle \cd_x v, w \rangle &= - \langle v, \cd_x w \rangle;
		\end{align*}
		\item Properties corresponding to the symmetry of the identity map and $ \partial_x^2 $:
		\begin{align*}
			\langle \fa_x v , w \rangle &=  \langle v , \ba_x w \rangle, &
			\langle \cd[2]_x v, w \rangle &= \langle v, \cd[2]_x w \rangle;
		\end{align*}
	\end{enumerate}
\end{lemma}

As the matrix representations of the above-mentioned spatial difference operators are circulant,
their common eigenvectors can be written in the form
\[ \frac{1}{\sqrt{K}} \begin{pmatrix} 1 & \exp \frac{\pi \im j}{K} & \exp \frac{2 \pi \im j }{K} & \cdots & \exp \frac{(K-1)\pi \im j}{K} \end{pmatrix}^{\top}  \qquad (j=1,\dots,K), \]
where $ \im $ is the imaginary unit, and their corresponding eigenvalues can be explicitly written as shown in the lemma below.
These eigenvalues are important because they provide us with several upper bounds.

\begin{lemma}\label{lem_eig}
	The $j$th eigenvalues of each operator can be written as
	\begin{align*}
		\lambda_j \left( \fd_x \right) &=  \frac{ 2 \im \sin \frac{\pi j}{K} }{\Delta x} \exp \frac{\pi \im j}{K},&
		\lambda_j \left( \cd_x \right) &=  \frac{ 2 \im \sin \frac{2 \pi j}{K} }{\Delta x} ,\\
		\lambda_j \left( \cd[2]_x \right) &= - \left( \frac{ 2 \sin \frac{ \pi j}{K} }{\Delta x} \right)^2 ,&
		\lambda_j \left( \fa_x \right) &=  \cos \frac{\pi j}{K}  \exp \frac{\pi \im j}{K}.
\end{align*}
\end{lemma}

Let us introduce the discrete Sobolev space $ H^1_K (\Sb) = \left( \RR^K , \| \cdot \|_{H^1_K(\Sb)} \right) $,
where the associated discrete Sobolev norm is defined as
\[ \| v \|_{H^1_K (\Sb)} := \left( \| v \|^2 + \left\| \fd_x v  \right\|^2  \right)^{\frac{1}{2}}. \]

\subsection{Moore--Penrose inverse}

Throughout this paper, we use the Moore--Penrose inverse of the difference operators
(see, e.g., \cite{IsraelGreville2003} for details on generalized inverses):
\begin{align*}
	\ifd_x &:= \pinv{\Big( \fd_x \Big)}, &
	\ibd_x &:= \pinv{\Big( \bd_x \Big)}, &
	\icd_x &:= \pinv{\Big( \cd_x \Big)}, &
	\icd[2]_x &:= \pinv{\Big( \cd[2]_x \Big)}. &
\end{align*}
Then, owing to the standard property of the Moore--Penrose inverse, the following useful lemma holds.

\begin{lemma}
	\[ \fd_x \ifd_x = \ifd_x \fd_x = \bd_x \ibd_x = \ibd_x \bd_x = \cd[2]_x \icd[2]_x = \icd[2]_x \cd[2]_x = P \]
	holds, where $ P $ is the orthogonal projector onto the set $ \{ v \in \RR^K \mid \sum_{k=1}^K v_k \Delta x = 0 \} $.
\end{lemma}

\begin{proof}
	Recall that, for the Moore--Penrose inverse $A^{\dagger} $ of a linear operator $A$,
	$ A^{\dagger} A $ and $ A A^{\dagger} $ are the orthogonal projectors onto $ \Car (A) := ( \Null( A) )^{\perp} $ and $ \Range (A) $, respectively.
	Therefore, it is sufficient to confirm that
	\[ \Range (\fd_x) = \Car (\fd_x) = \Range (\bd_x) = \Car (\bd_x) = \Range (\cd[2]_x) = \Car (\cd[2]_x) = \Range (P), \]
	which can be verified by their eigenvalues (see Lemma~\ref{lem_eig}). 
\end{proof}

The upper bound of the operator norm of $ \ifd_x $ plays an important role in the mathematical analysis (Section~\ref{sec_anal}).
Here, $ \| A \|_2 $ denotes the operator norm of a linear operator $ A $ defined as $ \| A \|_2 = \max_{ \| v \|_2 = 1 } \| A v \|_2 $.

\begin{lemma}\label{lem_bound_pinv_fd}
	If $K \ge 2$, $ \| \ifd_x \|_2 = \| \ibd_x \|_2 \le L / 4 $ holds.
\end{lemma}

\begin{proof}
	To estimate the upper bound of the operator norm $ \| \ifd_x \|_2 $ of $ \ifd_x $,
	we should estimate the lower bound of the square root of the nonzero eigenvalues of $ ( \fd_x )^{\ast} \fd_x $, where the superscript $ * $ denotes the adjoint.
	By using Lemma~\ref{lem_eig}, we can see that
		\begin{align*}
	\min_{ j = 1,\dots,K-1} \sqrt{\lambda_j \left( ( \fd_x )^{\ast} \fd_x  \right)}
	&= \min_{ j = 1,\dots,K-1} \sqrt{\lambda_j \left( - \cd[2]_x \right)}
	= \min_{j=1,\dots,K-1} \left| \frac{ 2 \sin \frac{\pi j }{K} }{\Delta x} \right| 
	= \frac{ 2 \sin \frac{\pi }{K} }{\frac{L}{K}} = \frac{2\pi}{L} \frac{ \sin \frac{\pi}{K} }{\frac{\pi}{K}}.
	\end{align*}
	Since the sinc function $ \sin x / x$ is monotone decreasing in $ [0,\pi] $,
	the right-hand side is monotone increasing for $K$, which proves the lemma.
\end{proof}

\subsection{Important inequalities}

The discrete counterpart of the Poincar\'e--Wirtinger inequality (Lemma~\ref{lem_PW}) will be used to prove the $L^{\infty} $ bound of the numerical solution (Theorem~\ref{thm_snb_d}).

\begin{lemma}[Discrete Poincar\'e--Wirtinger inequality (\protect{cf. \cite[Lemma~3.3]{FM2011}})]\label{lem_dPW}
	For any $ v \in H^1_K (\Sb) $,
	\begin{equation}\label{ineq_dPW}
		\| v - \bar{v} \mathbf{1} \|_{\infty} \le \sqrt{L} \left\| \fd_x v \right\|
	\end{equation}
	holds, where $ \bar{v} := (1/L) \sum_{k=1}^K v_k \Delta x $ and $ \mathbf{1} := (1,\dots, 1)^{\top} $.
\end{lemma}

To conduct mathematical analysis,
we frequently use the following standard equalities and inequalities.

\begin{lemma}[Standard equalities and inequalities]\label{lem_ineqs}
	The following equalities and inequalities hold:
	\begin{align*}
		\| v * w \| &\le \frac{1}{\sqrt{\Delta x}} \| v \| \|w\|, &
		\left\langle v , w \right\rangle &\le \frac{\|v \|^2 + \|w\|^2}{2}, \\
		\left\| \frac{v+w}{2} \right\|^2 &\le \| v \|^2 + \| w \|^2, &
		\left\| P \right\|_2 &= 1, \\
		\left\| \ba_x \right\|_2 &= \left\| \fa_x \right\|_2 = 1, &
		\left\| \bd_x \right\|_2 &= \left\| \fd_x \right\|_2 \le \frac{2}{\Delta x}.
	\end{align*}
	Here, $ v*w $ denotes the Hadamard product of $ v$ and $ w $, defined as $ (v * w)_k := v_k w_k $.
\end{lemma}

The discrete Gronwall lemma is a basis of the error estimate for the numerical method for evolutionary equations.

\begin{lemma}[Discrete Gronwall lemma~\protect{\cite[Lemma~1]{L1960}}]\label{lem_dG}
	Let $ \ed{m}{} $ be nonnegative sequences and $ \rhod{m}{} $ be nonnegative and nondecreasing sequences {\rm (}$ m = 0 , 1, \dots ,M ${\rm )}.
	If there exists $ c > 0 $ satisfying
	\[ \ed{m}{} \le \rhod{m}{} + c \Delta t \sum_{j=0}^{m-1} \ed{m}{} \]
	for all $ m = 0,1,\dots, M $, then
	\[ \ed{m}{} \le \rhod{m}{} \exp \left( c m \Delta t \right) \]
	holds for all $ m = 0,1,\dots,M $.
\end{lemma}

To achieve a better estimate in Section~\ref{subsec_ue},
we use the discrete counterpart of a special case of \cite[Lemma~A1]{K1975}
(as mentioned in \cite{K1975}, this lemma is a restatement of several propositions presented in \cite[Section~9]{Palais1968}).
To prove it, we use the discrete Sobolev lemma below.

\begin{lemma}[Discrete Sobolev lemma~\protect{\cite[Lemma~3.2]{FM2011}}]\label{lem_dS}
	For any $ v \in H^1_K (\Sb) $,
	\[ \| v \|_{\infty} \le \hat{L} \| v \|_{H^1_K (\Sb)} \]
	holds, where $ \hat{L} = \sqrt{2} \max \{ 1 / \sqrt{L}, \sqrt{L} \} $.
\end{lemma}

Now, we introduce the lemma inspired by~\cite[Lemma~A1]{K1975}.

\begin{lemma}\label{lem_Kato}
	For any $ v \in H^1_K (\Sb) $ and $ w \in L^2_K (\Sb) $,
	\begin{equation}
		\| v * w \| \le \hat{L} \| v \|_{H^1_K (\Sb)} \| w \|
	\end{equation}
	holds.
\end{lemma}

\begin{proof}
	Owing to the discrete Sobolev lemma~(Lemma~\ref{lem_dS}), we see that
	\begin{align*}
		\| v * w \|^2
		&= \sum_{k=1}^K v_k^2 w_k^2 \Delta x
		= \left\| v \right\|_{\infty}^2 \sum_{k=1}^K w_k^2 \Delta x
		\le \left( \hat{L} \| v \|_{H^1_K(\Sb)}  \right)^2 \| w \|^2,
	\end{align*}
	which proves the lemma.
\end{proof}

In Section~\ref{subsec_ue}, we use the above-mentioned lemma in the following form.

\begin{corollary}\label{cor_pinv}
	For any $ v , w \in L^2_K (\Sb) $,
	\[ \left\| ( \ibd_x v ) * w \right\| \le C \| v \| \| w \| \]
	holds, where $C := (\hat{L}/4 ) \sqrt{L^2 + 16}$.
\end{corollary}

\begin{proof}
	By using Lemma~\ref{lem_bound_pinv_fd}, we obtain
	\[ \left\| \ibd_x v \right\|_{H^1_K(\Sb)}^2 = \left\| \ibd_x v \right\|^2 + \left\| \fd_x \ibd_x v \right\|^2 \le \frac{L^2}{16} \| v \|^2 + \| v\|^2 \le \frac{L^2 + 16}{16} \| v \|^2. \]
	Therefore, by using Lemma~\ref{lem_Kato}, we can estimate $ \left\| ( \ibd_x v ) * w \right\| $ as follows:
	\[ \left\| ( \ibd_x v ) * w \right\| \le \hat{L} \left\| \ibd_x v \right\|_{H^1_K(\Sb)} \| w \| \le \hat{L} \frac{\sqrt{L^2 + 16}}{4} \| v \| \| w \|. \]
\end{proof}

\section{$L^{\infty} $ bound and blow up phenomena}
\label{sec_conti}

In this section, we briefly review several features of the mHS equation~\eqref{eq_mHS} investigated by Fu and Yin~\cite{FY2010} and Li and Yin~\cite{LY2017_1}.
First, several invariants are introduced,
and they are used to show the $ L^{\infty} $ bound of the solution.
Finally, the results on blow up phenomena are summarized.

\subsection{Several invariants and $L^{\infty} $ bound}

As mentioned in the Introduction, each solution $u$ of~\eqref{eq_mHS} satisfies the implicit constraint $ \mathcal{F} (u(t)) = 0 $ (see \eqref{eq_ic}),
which can be confirmed by integrating both sides of the mHS equation.
On the other hand, as also mentioned in the Introduction, $ \mathcal{H} (u(t)) = \mathcal{H} (u_0) $ holds for each solution~$u$.
Moreover, it should be noted that, owing to the above-mentioned invariants,
each solution $u$ of~\eqref{eq_mHS} satisfies
\begin{equation}~\label{eq_av}
	\int_{\Sb} u (t,x) \rd x = - \frac{1}{4\omega} \int_{\Sb} \left( u_x (t,x) \right)^2 \rd x = - \frac{1}{4\omega} \int_{\Sb} \left( u_0 (x) \right)_x^2 \rd x = \int_{\Sb} u_0 (x) \rd x .
\end{equation}

The Poincar\'e--Wirtinger inequality below reveals the $ L^{\infty} $ bound of the solution
owing to these conservation laws.

\begin{lemma}[\protect{Poincar\'e--Wirtinger inequality (cf. \cite[Chapter 8]{Brezis2011})}]\label{lem_PW}
	\begin{equation}\label{ineq_PW}
		\| w - \bar{w} \mathbf{1} \|_{\infty} \le \sqrt{L} \| w_x \|_2
	\end{equation}
	holds for all $ w \in H^1 (\Sb) $, where $ \bar{w} = (1/L) \int_{\Sb} w (x) \rd x $ and $ \mathbf{1} $ denotes a constant function satisfying $ \mathbf{1} (x) = 1 \ ( x \in \Sb)  $.
\end{lemma}

\begin{proposition}[$L^{\infty} $ bound (cf. \protect{\cite[Lemma~3.1]{LY2017_1}})]\label{prop_snb}
	All solutions $ u $ of mHS equation~\eqref{eq_mHS} satisfy
	\[ \| u(t) \|_{\infty} \le L \sqrt{| 4 \omega h |} + | h| \]
	for $ t \in (0,T) $, where $ h := (1/L ) \int_{\Sb} u_0 (x) \rd x $.
\end{proposition}

\begin{proof}
	Owing to the Poincar\'e--Wirtinger inequality~\eqref{ineq_PW}, we see that
	\begin{align*}
		\| u(t) \|_{\infty}
		&\le \left\| u(t) - h \mathbf{1} \right\|_{\infty} +  \left| h \right| \le \sqrt{L} \| u_x (t) \|_2 + | h |
		= L \sqrt{| 4 \omega h |} + | h|,
	\end{align*}
	which proves the proposition. 
\end{proof}

\subsection{Blow up phenomena}
\label{subsec_bup}

Li and Yin~\cite{LY2017_1} showed a blow up criterion in $H^2 (\Sb) $ as follows.

\begin{proposition}[\protect{\cite[Lemma~3.2]{LY2017_1}}]\label{prop_bu_ux}
	Let $u_0 \in H^2(\Sb) $ be an initial condition satisfying $ \mathcal{F} (u_0) = 0 $.
	Then, the corresponding solution blows up at $ T < + \infty $ with respect to $ H^2 (\Sb) $ if and only if 
	\[ \limsup_{t \uparrow T} \{ - \inf_{x \in \Sb} u_{x} (t,x) \} = + \infty. \]
\end{proposition}

Furthermore, Fu and Yin~\cite{FY2010} showed a blow up criterion in $ H^s (\Sb) $ ($ s > 5/2 $) as follows.

\begin{proposition}[\protect{\cite[Theorem~3.1]{FY2010}}]\label{prop_bu_uxx}
	Let $ u_0 \in H^s (\Sb) $ be an initial condition {\rm (}$ s > 5/2 ${\rm )} satisfying $ \mathcal{F} (u_0) = 0 $.
	Then, the corresponding solution blows up at $ T < + \infty $ with respect to $ H^s (\Sb) $ if and only if 
	\[ \limsup_{t \uparrow T} \| u_{xx} (t) \|_{\infty} = + \infty .\]
\end{proposition}

Summing up, for an initial condition $ u_0 \in H^r (\Sb) $ {\rm (}$r>5/2${\rm )} and $ s \in (5/2,r] $, the blow up time $ T^{(s)} $ with respect to $ H^s(\Sb) $
satisfies $ T^{(r)} = T^{(s)} \le T^{(2)} $.
However, thus far, it has not been proved whether the inequality is strict.

On the other hand, Li and Yin~\cite{LY2017_1} found that
there is some initial condition such that the corresponding solution blows up in finite time in the sense of Proposition~\ref{prop_bu_ux}.

\begin{proposition}[\protect{\cite[Theorem~3.4]{LY2017_1}}]\label{prop_bu}
	Let $u_0 \in H^s (\Sb) $ be an initial condition {\rm (}$ s > 3 ${\rm )} satisfying $ \mathcal{F} (u_0) = 0 $ and $ |h| \ge 4 L^2 | \omega | $.
	Then, the corresponding solution satisfies $  \limsup_{t \uparrow T} \{ - \inf_{x \in \Sb} u_{x} (t,x) \} = + \infty $ for finite $T > 0 $, i.e., it blows up in finite time.
\end{proposition}

However, we do not know whether the blow up in the sense of Proposition~\ref{prop_bu_uxx} occurs.
This will be numerically investigated later on.

\section{Derivation of a stable numerical scheme}
\label{sec_des}

As shown in Proposition~\ref{prop_snb},
the $ L^{\infty} $ bound derives from the fact that
$ \int_{\Sb} u_x^2 (t,x) \rd x $ and $ \int_{\Sb} u(t,x) \rd x $ remain the same as time passes,
which is ensured by invariants $ \mathcal{F} $ and $ \mathcal{H} $.
Therefore, to replicate the $L^{\infty} $ bound,
we derive a numerical scheme rigorously preserving $ \mathcal{F} $ and $ \mathcal{H} $.

Section~\ref{subsec_reform} is devoted to the reformulation for achieving the discrete conservation.
Then, we derive a conservative finite difference scheme and prove its stability in Section~\ref{subsec_design}.

\subsection{Reformulation for invariant preservation}
\label{subsec_reform}

As mentioned in the Introduction, to achieve preservation of $ \mathcal{F} $ and $ \mathcal{H} $,
some special treatment is indispensable.
Toward this end,
we consider the underdetermined form~\eqref{eq_mHS2}.
Then, it is easy to observe that, for sufficiently smooth solutions,
the original problem~\eqref{eq_mHS} is equivalent to the initial value problem
\begin{equation}\label{eq_ivp_mHS2}
	\begin{cases}
		u_{txx} + \frac{1}{2} \left(u^2 \right)_{xxx} = 2 \omega u_x + \frac{1}{2} \left( u_x^2 \right)_x \qquad &( t \in (0,T) , x \in \Sb),\\
		\mathcal{F} (u(t)) = 0 \qquad & ( t \in (0,T) ),\\
		u(0,x) = u_0 (x) & (x \in \Sb) . \\
	\end{cases}
\end{equation}

On the other hand, as stated by Miyatake, Cohen, Furihata, and Matsuo~\cite{MCFM2017},
all solutions $u$ of~\eqref{eq_mHS2} (not~\eqref{eq_ivp_mHS2}) satisfy $ \mathcal{H} (u(t)) = \mathcal{H} (u(0)) $ (see also~Lenells~\cite{L2008}).
To confirm this, we rewrite~\eqref{eq_mHS2} as
\begin{equation}\label{eq_mHS2_con}
	u_{txx} = \mathcal{A} (u(t)) u,
\end{equation}
where the linear operator $ \mathcal{A} (v) $ is defined as $  \mathcal{A} (v) := (\omega - v_{xx}) \partial_x + \partial_x ( \omega - v_{xx} ) $,
and we see that
\begin{align*}
	\frac{\rd}{\rd t} \mathcal{H} (u(t))
	&= \int_{\Sb} u_x u_{tx} \rd x = - \int_{\Sb} u u_{txx} \rd x
	= - \int_{\Sb} u \mathcal{A} (u(t)) u \rd x = 0,
\end{align*}
where the last equality comes from the skew-symmetry of the linear operator $ \mathcal{A} (u(t)) $.

\begin{remark}
	The spatial differential operator $ \partial_x$ in the linear operator $ \mathcal{A} (u(t)) $ operates on $ ( \omega - u_{xx} ) u $.
	Although this convention may seem confusing,
	we employ this notation to emphasize the importance of the skew-symmetry of the linear operator $\mathcal{A} (u(t))$.
\end{remark}

\begin{remark}
	In view of the regularity, the initial value problems~\eqref{eq_mHS} and~\eqref{eq_ivp_mHS2} are certainly different.
	However, as the main objective of this paper is to show the convergence rate of our numerical scheme,
	we restrict ourselves to sufficiently smooth solutions so that this discrepancy will not cause any restriction.
\end{remark}

\subsection{Derivation of a stable numerical scheme for mHS equation}
\label{subsec_design}

To replicate the preservation of $\mathcal{H} $, we employ the discretization of ~\eqref{eq_mHS2_con} proposed by Miyatake, Cohen, Furihata, and Matsuo~\cite{MCFM2017}:
\begin{equation}\label{eq_dvdm_mHS2_con}
	\cd[2]_x \fd_t \ud{m}{k} = \mathcal{A}_{\rd} \big( \ud{m+1/2}{} \big) \ud{m+1/2}{k},
\end{equation}
where $ \mathcal{A}_{\rd} \left( v \right) := \Big( \omega -  \cd[2]_x v \Big) \cd_x  + \cd_x \Big( \omega - \cd[2]_x v \Big) $ and $ \ud{m+1/2}{k} := \fa_t \ud{m}{k} $.
As mentioned in \cite{MCFM2017}, the conservation of the discrete counterpart $ \mathcal{H}_{\rd} $ of $ \mathcal{H} $ is ensured,
where
\begin{equation}\label{eq_norm_d}
	\mathcal{H}_{\rd} (v) := \frac{1}{2} \sum_{k=1}^K ( \fd_x v_k )^2 \Delta x .
\end{equation}
Although it has been proved in \cite{MCFM2017}, we show it here for the readers' convenience:
\begin{align*}
	\fd_t \mathcal{H}_{\rd} \left( \ud{m}{} \right)
	&= \sum_{k=1}^K \left( \fa_t \fd_x \ud{m}{k} \right) \left( \fd_t \fd_x \ud{m}{k} \right) \Delta x 
	= - \sum_{k=1}^K \ud{m+1/2}{k} \left( \cd[2]_x \fd_t \ud{m}{k} \right) \Delta x \\
	&= - \sum_{k=1}^K \ud{m+1/2}{k} \mathcal{A}_{\rd} \big( \ud{m+1/2}{} \big) \ud{m+1/2}{k} \Delta x
	= 0,
\end{align*}
where the last equality is due to the skew-symmetry of the linear operator $  \mathcal{A}_{\rd} \big( \ud{m+1/2}{} \big) $.

The discretization~\eqref{eq_dvdm_mHS2_con} is obviously underdetermined, as is~\eqref{eq_mHS2_con}.
Hence, we should impose some ``constraint'' to make it uniquely solvable.
By employing some discrete counterpart of $ \mathcal{F} (u(t)) = 0 $ in \eqref{eq_ivp_mHS2} as such a ``constraint,''
we achieve the discrete preservation of $ \mathcal{F} $ and $ \mathcal{H} $.
In view of~\eqref{eq_av}, we define $ \mathcal{F}_{\rd} $ as
\begin{equation}\label{eq_ic_d}
	\mathcal{F}_{\rd} (v) := \sum_{k=1}^K \left( 2 \omega v_k + \frac{1}{2} \left( \fd_x v_k \right)^2 \right) \Delta x,
\end{equation}
as a discrete counterpart of $ \mathcal{F} $.

Thus, our numerical scheme is the following one-step method:
\begin{equation}\label{eq_proposed}
	\begin{cases}
		\cd[2]_x \fd_t \ud{m}{k} = \mathcal{A}_{\rd} \big( \ud{m+1/2}{} \big) \ud{m+1/2}{k} \quad & (m=0,1,\dots,M-1;k \in \mathbb{Z}),\\
		\mathcal{F}_{\rd} \big( \ud{m+1}{} \big) = \mathcal{F}_{\rd} \big( \ud{0}{} \big) & (m=0,1,\dots,M-1),\\
		\ud{0}{k} = u_0 ( k \Delta x ) & (k \in \mathbb{Z}),\\
		\ud{m}{k+K} = \ud{m}{k} & ( m = 0,1,\dots,M-1, k \in \mathbb{Z} ).
	\end{cases}
\end{equation}
Here, since $ \mathcal{F}_{\rd} \big( \ud{0}{} \big) \neq 0 $ even when $ \mathcal{F} (u_0) = 0 $,
we use the constraint $ \mathcal{F}_{\rd} \big( \ud{m}{} \big) = \mathcal{F}_{\rd} \big( \ud{0}{} \big) $ as a discrete counterpart of $ \mathcal{F} (u(t)) = \mathcal{F} (u_0) = 0 $ (see Remark~\ref{rem_equiv}).

It should be noted that, owing to discrete invariants $ \mathcal{H}_{\rd} $ and $ \mathcal{F}_{\rd} $,
each solution $\ud{m}{k} $ of~\eqref{eq_proposed} satisfies
\begin{align*}
	\sum_{k=1}^K \ud{m}{k} \Delta x
	&= \frac{1 }{2 \omega } \left( - \frac{1}{2} \sum_{k=1}^K \left( \fd_x \ud{m}{k} \right)^2 \Delta x +  \mathcal{F}_{\rd} \big( \ud{0}{} \big) \right) \\
	&= \frac{1 }{2 \omega } \left( - \frac{1}{2} \sum_{k=1}^K \left( \fd_x \ud{0}{k} \right)^2 \Delta x +  \mathcal{F}_{\rd} \big( \ud{0}{} \big) \right)
	= \sum_{k=1}^K \ud{0}{k} \Delta x .
\end{align*}

Therefore, the discrete Poincar\'e--Wirtinger inequality (Lemma~\ref{lem_dPW}) reveals the discrete $ L^{\infty} $ bound of the numerical solution.

\begin{theorem}[$L^{\infty} $ bound]\label{thm_snb_d}
	All solutions $ \ud{m}{k} $ of \eqref{eq_proposed} satsify
	\[ \big\| \ud{m}{} \big\|_{\infty} \le L \sqrt{| 4 \omega  \hd |} + | \hd | \]
	for $ m = 0,1,\dots,M $, where $ \hd := (1 / L ) \sum_{k=1}^K \ud{0}{k} \Delta x $.
\end{theorem}

\begin{proof}
	Owing to the discrete Poincar\'e--Wirtinger inequality~\eqref{ineq_dPW}, we see that
	\begin{align*}
		\big\| \ud{m}{} \big\|_{\infty}
		&\le \big\| \ud{m}{} - \hd \mathbf{1} \big\|_{\infty} +  \left| \hd \right| \le \sqrt{L} \big\| \fd_x \ud{m}{} \big\|_2 + | \hd |
		= L \sqrt{| 4 \omega  \hd |} + | \hd |,
	\end{align*}
	which proves the theorem. 
\end{proof}

\begin{remark}\label{rem_equiv}
	Owing to the definition of the discrete constraint in the scheme~\eqref{eq_proposed},
	our method can be regarded as a discretization of the initial value problem
	\begin{equation}\label{eq_mHS_ode}
		\begin{cases}
			u_{txx} + \frac{1}{2} \left(u^2 \right)_{xxx} = 2 \omega u_x + \frac{1}{2} \left( u_x^2 \right)_x \qquad & (t \in (0,T), x \in \Sb ), \\
			\mathcal{F}(u(t)) = \mathcal{F} (u_0 ) & (x \in \Sb ), \\
			u (0,x) = u_0 (x) & (x \in \Sb),
		\end{cases}
	\end{equation}
	even when $ \mathcal{F} (u_0) \neq 0 $.
	Note that the case $ \mathcal{F} (u_0) = 0 $ corresponds to the target problem~\eqref{eq_mHS}.
	Moreover, although we focus on the case $ \mathcal{F} (u_0) = 0 $ in this paper,
	the mathematical analysis in the following sections is valid for the (artificial) general case~\eqref{eq_mHS_ode}.
	
	As shown in Section~\ref{sec_relation} in the Appendix of the present paper,
	our method actually coincides with the scheme of Miyatake, Cohen, Furihata, and Matsuo~\cite{MCFM2017}	as a numerical method for~\eqref{eq_mHS_ode}.
	Therefore, the results reported in the present paper are also valid for their numerical scheme.
	In other words, the scheme of Miyatake, Cohen, Furihata, Matsuo~\cite{MCFM2017} is uniquely solvable and its numerical solutions are convergent to appropriate traveling wave solutions.
\end{remark}

\section{Unique existence and convergence rate}
\label{sec_anal}

In this section, we prove the unique existence of the numerical solution of the scheme~\eqref{eq_proposed}
and show its convergence rate.

\subsection{Unique existence of the numerical solution}
\label{subsec_ue}

Here, we show the unique existence of the numerical solution of~\eqref{eq_proposed}.
However, if we naively use the contraction mapping theorem,
we obtain some severe sufficient condition for unique existence owing to a large number of difference operators.
In this section, we show that
the appropriate reformulation~\eqref{eq_proposed_v_int} and Corollary~\ref{cor_pinv} (see Remark~\ref{rem_pinv})
enable us to achieve a mild sufficient condition $ \Delta t = O (\Delta x ) $.

Toward this end, we consider a reformulation
by introducing a new variable $ \vd{m}{k} := \bd_x \ud{m}{k} $.
Then, $ \ud{m}{k} = \ibd_x \vd{m}{k} + \hd $ holds owing to the invariant $ \sum_{k=1}^K \ud{m}{k} \Delta x $
(recall that $ \hd = (1/L ) \sum_{k=1}^K \ud{0}{k} \Delta x $).
Thus, we consider the reformulation
\begin{align*}
	\fd_x \fd_t \vd{m}{k}
	&= 2 \omega \fa_x \vd{m+1/2}{k} - \frac{1}{2} \fd_x \left( \vd{m+1/2}{k} \right)^2 
	 - \cd_x \left( \left( \ibd_x \vd{m+1/2}{k} + \hd \right) \left( \fd_x \vd{m+1/2}{k} \right) \right).
\end{align*}
Since $ \fd_t \vd{m}{} \in \Car ( \fd_x )  $ holds, operating with $ \ifd_x $ on both sides yields
 \begin{align}
 	\fd_t \vd{m}{k}
 	&= 2 \omega \ifd_x \fa_x \vd{m+1/2}{k} - \frac{1}{2} P \left( \vd{m+1/2}{k} \right)^2 
 	 - P \ba_x \left( \left( \ibd_x \vd{m+1/2}{k} + \hd \right) \left( \fd_x \vd{m+1/2}{k} \right) \right) \label{eq_proposed_v_int}
 \end{align}
(recall that $ P $ is the orthogonal projector onto the set of zero-mean vectors).

The following lemma shows the validity of this reformulation.

\begin{lemma}\label{lem_equiv_d}
	If $ \ud{m}{k} $ is a solution of the scheme~\eqref{eq_proposed},
	then $ \vd{m}{k} := \bd_x \ud{m}{k} $ solves~\eqref{eq_proposed_v_int}.
	Conversely, if $ \vd{m}{k} $ is a solution of~\eqref{eq_proposed_v_int} with the initial condition $ \vd{0}{k} = \bd_x \ud{0}{k} $,
	then $ \ud{m}{k} := \ibd_x \vd{m}{k} + \hd $ solves~\eqref{eq_proposed}.
\end{lemma}

\begin{proof}
	The former part has already been confirmed above.
	For the latter part, we assume that $ \vd{m}{k} $ is the solution of~\eqref{eq_proposed_v_int} satisfying $ \vd{0}{k} = \bd_x \ud{0}{k} $ and
	define $ \ud{m}{k} = \ibd_x \vd{m}{k} + \hd $.
	Then, it is sufficient to prove that $ \bd_x \ud{m}{k} = \vd{m}{k} $ and $ \fd_x \ifd_x \fa_x \vd{m+1/2}{k} = \fa_x \vd{m+1/2}{k} $.
	In other words, we should confirm that $ \sum_{k=1}^K \vd{m}{k} \Delta x = 0 $ holds for any $ m = 0,1,\dots,M $.
	Since $ \sum_{k=1}^K \vd{0}{k} = 0 $ is satisfied by definition,
	it suffices to confirm that $ \fd_t \vd{m}{k} $ is a zero-mean vector for any $m$,
	which is obvious by $ \Range ( \ifd_x ) = \Range (P) = \{ v \in \RR^K \mid \sum_{k=1}^K v_k \Delta x = 0 \}$. 
\end{proof}

To prove the unique existence of the solution of~\eqref{eq_proposed_v_int},
we use the contraction mapping theorem.
Toward this end, we introduce
\begin{align}
	\phi_{v} (w) &= v + \omega \Delta t \ifd_x \fa_x w - \frac{\Delta t}{4} P \psi (w), \label{eq_fpi} \\
	\psi (w) &= w^2 + 2 \ba_x \left( \left( \ibd_x w + \hd \mathbf{1} \right) * \left( \fd_x w \right) \right) , \notag
\end{align}
where $ w^2 := w * w $.
Then, it holds that
\[ \text{$ w^{\ast} $ is a fixed point of $ \phi_{\vd{m}{}} $} \iff \text{$ \vd{m+1}{} = 2 w^{\ast} - \vd{m}{} $ is a solution of \eqref{eq_proposed_v_int}}. \]

The following lemma provides us with a sufficient condition for
ensuring that $ \phi_v $ maps some closed ball into itself.

\begin{lemma}\label{lem_ball}
	Let $ B (pr) $ be the ball defined as $ B (pr) = \{ w \mid \| w \| \le pr \}$ {\rm (}$p>1$, $ r = \| v \|_2 ${\rm )}.
	If mesh sizes $ \Delta t $ and $ \Delta x $ satisfy
	\[ \Delta t \le \epsilon_1 (p,r):= \frac{4 (p-1) \Delta x}{ p} \left(  \left| \omega \right| L \Delta x + p r \sqrt{\Delta x} + 4 \left| \hd \right| + 4C p r \right)^{-1} , \]
	then $ \phi_{v} (B (pr)) \subseteq B (pr) $ holds.
	Here, $ C \in \RR $ is a constant defined in Corollary~\ref{cor_pinv}.
\end{lemma}

\begin{proof}
	Let $ w $ be an element of $B(pr) $.
	Then, by using Lemmas~\ref{lem_ineqs} and \ref{lem_bound_pinv_fd} and Corollary~\ref{cor_pinv}, we see that
	\begin{align*}
		\| \psi (w) \|
		&\le \left\| w^2 \right\| + 2 \left\| \left( \ibd_x w \right) * \left( \fd_x w \right) \right\| +2 \left| \hd \right| \left\| \fd_x w \right\| 
		\le \frac{1}{\sqrt{\Delta x }} p^2 r^2 + 2 C \| w \| \left\| \fd_x w \right\| + \frac{4\left| \hd \right|}{\Delta x} p r \\
		&\le \frac{1}{\sqrt{\Delta x }} p^2 r^2 + 2 C \frac{2}{\Delta x} p^2 r^2 + \frac{4\left| \hd \right|}{\Delta x} p r 
		= \frac{pr}{\Delta x} \left( p r \sqrt{\Delta x} + 4 \left| \hd \right| + 4 C p r \right).
	\end{align*}
	Therefore, the norm $ \| \phi_v (w) \| $ can be estimated by
	\begin{align*}
		\| \phi_v (w) \|
		&\le \| v \| + \left| \omega \right| \Delta t \frac{L}{4} \|w\| + \frac{\Delta t}{4} \| \psi (w) \| 
		\le r + \frac{p r \Delta t}{4 \Delta x} \left( \left| \omega \right| L \Delta x + p r \sqrt{\Delta x} + 4 \left| \hd \right| + 4 C p r \right).
	\end{align*}
	When the assumption is satisfied, the right-hand side is bounded by $pr$, which proves the lemma. 
\end{proof}

\begin{remark}\label{rem_pinv}
	It should be noted that Corollary~\ref{cor_pinv} provides us with $ \epsilon_1 (p,r) = O (\Delta x) $.
	In fact, if we employ the standard argument using $ \| v * w \| \le ( 1 / \sqrt{\Delta x} ) \| v \| \| w \| $ for the term
	$ \left\| \left( \ibd_x w \right) * \left( \fd_x w \right) \right\| $,
	the term $ O ( ( \Delta x )^{-3/2} ) $ appears. Hence, we have $ \epsilon_1 (p,r) = O( (\Delta x)^{3/2} ) $.
\end{remark}

The following lemma is to show the sufficient condition for $ \phi_v $ being a contraction mapping.

\begin{lemma}\label{lem_contraction}
	If mesh sizes $ \Delta t $ and $ \Delta x $ satisfy
	\[ \Delta t < \epsilon_2 (p,r):= 4 \Delta x \left( \left| \omega \right| L \Delta x + 2 p r \sqrt{\Delta x} +  4 \left| \hd \right| + 8 C p r \right)^{-1}, \]
	then $ \phi_v : B(pr) \to B(pr) $ is a contraction mapping.
	Here, $ C \in \RR $ is a constant defined in Corollary~\ref{cor_pinv}.
\end{lemma}

\begin{proof}
	Let $ w_1 , w_2 $ be elements of $B (p,r)$.
	By using
	\begin{align*}
		\left\| w_1^2 - w_2^2 \right\|
		&= \left\| \left( w_1 + w_2 \right) * \left( w_1 - w_2 \right) \right\|
		\le \frac{2pr}{\sqrt{\Delta x}} \| w_1 - w_2 \|
	\end{align*}
	and
	\begin{align*}
		& \left\| \left( \ibd_x w_1 + \hd \mathbf{1} \right) * \left( \fd_x w_1 \right) - \left( \ibd_x w_2 + \hd \mathbf{1} \right) * \left( \fd_x w_2 \right) \right\| \\
		\le{}& \left\| \left( \ibd_x w_1 \right) * \left( \fd_x w_1 \right) - \left( \ibd_x w_2 \right) * \left( \fd_x w_2 \right) \right\| + \left| \hd \right| \| \fd_x ( w_1 - w_2 ) \| \\
		\le{}& \left\| \left( \ibd_x \frac{w_1 + w_2}{2} \right) * \left( \fd_x (w_1-w_2) \right) \right\| + \left\| \left( \ibd_x (w_1-w_2) \right) * \left( \fd_x \frac{w_1+w_2}{2} \right) \right\| 
		 + \frac{2 \left| \hd \right|  }{\Delta x} \| w_1 - w_2\| \\
		\le{}& C \left\| \frac{w_1+w_2}{2} \right\| \left\| \fd_x (w_1 - w_2) \right\| +C \left\| w_1 - w_2 \right\| \left\| \fd_x \frac{w_1 + w_2}{2} \right\| + \frac{2 \left| \hd \right|  }{\Delta x} \| w_1 - w_2\| \\
		\le{}& \frac{4 C pr + 2 \left| \hd \right|}{\Delta x} \| w_1 - w_2 \|,
	\end{align*}
	we obtain
	\begin{align*}
		\| \phi_v (w_1) - \phi_v (w_2) \|
		&\le \left( \left| \omega \right| \frac{L}{4} \Delta t  + \frac{\Delta t}{4} \left( \frac{2pr}{\sqrt{\Delta x}} + 2 \frac{4 C pr + 2 \left| \hd \right|}{\Delta x}   \right) \right) \| w_1 - w_2 \|\\
		&\le \frac{\Delta t}{4 \Delta x} \left( \left| \omega \right| L \Delta x + 2 p r \sqrt{\Delta x} +  4 \left| \hd \right| + 8 C p r \right) \| w_1 - w_2 \|.
	\end{align*}
	Thus, under the assumption of the lemma, $ \phi_v : B(pr) \to B(pr) $ is a contraction mapping.
\end{proof}

By Lemmas~\ref{lem_equiv_d}, \ref{lem_ball}, and \ref{lem_contraction},
we obtain the following unique existence theorem.
It should be noted that, owing to the conservation law with respect to $ \mathcal{H}_{\rd} $,
$ \epsilon_1 (p,r) $ and $ \epsilon_2 (p,r) $ can be computed from the initial condition for any fixed $p > 1$.

\begin{theorem}\label{thm_ue}
	Let $ p $ and $r$ be real numbers satisfying $ p > 1 $ and $ r = \big\| \bd_x \ud{0}{} \big\| $.
	Then, if $ \Delta t $ satisfies $\Delta t < \min \{ \epsilon_1 (p,r) ,\epsilon_2 (p,r) \}$, the scheme~\eqref{eq_proposed} has a unique solution
	$ \ud{m}{k} \ ( m = 1, \dots, M; k=1,\dots ,K) $.
\end{theorem}

\begin{remark}\label{rem_dt}
	When we execute the scheme by using the fixed point iteration for~\eqref{eq_fpi},
	we should determine the step size $ \Delta t $ satisfying $ \Delta t < \min \{ \epsilon_1 (p,r) ,\epsilon_2 (p,r) \} $ for some $p>1$.
	As an extremely small $ \Delta t $ is not preferable, $ \Delta t $ should be chosen to be slightly below $ \max_{p \in (1,\infty)} \{ \min \{ \epsilon_1 (p,r) ,\epsilon_2 (p,r) \} \} $.
	This can be done by solving the one-dimensional nonlinear equation $ \epsilon_1 (p,r) = \epsilon_2 (p,r) $ when $ \Delta x $, $ \vd{0}{} $ and $ L $ are fixed.
	
	Nevertheless, there is another reasonable choice of $ \Delta t $.
	Since $ \epsilon_1 (2,r) < \epsilon_2 (2,r) $ holds, $ \Delta t \le \epsilon_1 ( 2 , r )  $ is a sufficient condition of unique solvability.
	Moreover, in addition to the fact that $ \epsilon_2 (p,r) $ is monotone decreasing in the regime $ p \in (1,\infty) $,
	we see that
	\begin{align*}
		\epsilon_1 (p,r) &\approx \frac{p-1}{p \left( \left| \hd \right| + C p r  \right)} \Delta x
	\end{align*}
	when $ \Delta x $ is sufficiently small, which implies that $ \epsilon_1 (p,r) $ is expected to be maximum around $ p = 1 + \sqrt{1+| \hd|/(Cr)} > 2$.
	Thus, if $ | \hd | / ( Cr) = r / (4 | \omega | C L ) $ is relatively small (i.e., the initial value $u_0$ is sufficiently calm; see also Proposition~\ref{prop_bu}),
	 the choice $ p = 2 $ (and accordingly, $ \Delta t \approx \epsilon_1 (2,r) $) is nearly optimal.
\end{remark}

\subsection{Convergence analysis}
\label{subsec_conv}

First, we estimate the local truncation error to conduct convergence analysis.
Among several reformulations of our scheme,
\eqref{eq_proposed_v_int} seems to be fit for convergence analysis
because it is in the standard form of evolutionary equations, i.e.,
there is no singular operator operating on $ \fd_t \vd{m}{k} $.
However, owing to the presence of nonlocal operators, such as $ \ifd_x $,
the estimation of the local truncation error turns out to be problematic.

Therefore, we deal with the local truncation error of the original form~\eqref{eq_proposed}.
Later, we show that it is sufficient to conduct convergence analysis owing to discrete conservation laws.

\begin{lemma}\label{lem_lte}
	Let $u$ be a solution of~\eqref{eq_mHS} satisfying $ u ( t) \in C^7 (\Sb) $ and $ u(\cdot, x) \in C^3 (0,T) $.
	Then, for any $ m = 1, \dots , M , k = 1,\dots, K $, and sufficiently small $ \Delta t $ and $ \Delta x $, $ \tud{m}{k} := u( m \Delta t , k \Delta x ) $ satisfies
	\[ \cd[2]_x \fd_t \tud{m}{k} = \mathcal{A}_{\rd} \big( \tud{m+1/2}{} \big) \tud{m+1/2}{k} + \taud{m+1/2}{k} , \]
	where $ \big\| \taud{m+1/2}{} \big\| \le b ( (\Delta x )^2 + (\Delta t )^2 ) $ for some positive constant $ b \in \RR$ depending only on the exact solution~$u$.
\end{lemma}

\begin{proof}
	The Taylor expansion around the point $ ( (m+1/2) \Delta t, k \Delta x ) $ yields the lemma. 
\end{proof}

Owing to the above-mentioned lemma, by introducing $ \ed{m}{k} := \ud{m}{k} - \tud{m}{k} $, we see that
\begin{equation}\label{eq_lerror}
	\cd[2]_x \fd_t \ed{m}{k} = 2 \omega \cd_x \ed{m+1/2}{k} - \left( \eta_k \big(\ud{m+1/2}{} \big)- \eta_k \big( \tud{m+1/2}{} \big) \right) - \taud{m}{k},
\end{equation}
where
\[ \eta_k (w) : = \left( \cd[2]_x w_k \right) \left( \cd_x w_k \right) + \cd_x \left( w_k  \left( \cd[2]_x w_k \right) \right). \]
Then, the following lemma holds.

\begin{lemma}\label{lem_db}
	When $ \Delta t $ and $ \Delta x $ are sufficiently small,
	\[ \big\| \fd_x \ed{m}{} \big\|^2 \le a \Delta t \big\| \ed{m}{} \big\|_{H^1_K (\Sb)}^2  + 2 a \Delta t \sum_{j=0}^{m-1} \big\| \ed{j}{} \big\|_{H^1_K (\Sb)}^2  + b L T ((\Delta x)^2 + (\Delta t)^2 )^2 \]
	holds for some positive constant $a \in \RR$ depending only on the exact solution $u$ and all $ m = 0,1,\dots, M $,
	where $ b $ is a positive constant in Lemma~\ref{lem_lte}.
\end{lemma}

\begin{proof}
	Since we see that
	\begin{align*}
		\fd_t \big\| \fd_x \ed{m}{} \big\|^2
		&= \left\langle  \fd_x \ed{m+1/2}{} ,  \fd_x \fd_t \ed{m}{} \right\rangle \\
		&= - \left\langle \ed{m+1/2}{} , \cd[2]_x \fd_t \ed{m}{} \right\rangle \\
		&= - \left\langle \ed{m+1/2}{} , 2 \omega \cd_x \ed{m+1/2}{} \right\rangle + \left\langle \ed{m+1/2}{}, \eta \big( \ud{m}{} \big) - \eta \big( \tud{m}{} \big) \right\rangle 
		 + \left\langle \ed{m+1/2}{} , \taud{m}{} \right\rangle \\
		&= \left\langle \ed{m+1/2}{}, \eta \big( \ud{m}{} \big) - \eta \big( \tud{m}{} \big) \right\rangle + \frac{\big\| \ed{m+1/2}{} \big\|^2 + \big\| \taud{m}{} \big\|^2 }{2},
	\end{align*}
	we cope with the first term on the right-hand side.
	In the remainder of this proof, we abbreviate $ \ed{m+1/2}{} $ by omitting the superscript $ (m+1/2) $
	(similar abbreviations are used for $ \ud{m+1/2}{}$ and $ \tud{m+1/2}{} $).
	Here, by using
	\begin{align*}
		&\left\langle e , \left( \cd[2]_x u \right) * \left( \cd_x u \right)  - \left( \cd[2]_x \tilde{u} \right) * \left( \cd_x \tilde{u} \right) \right\rangle \\
		={}& \left\langle e, \left( \cd[2]_x \frac{u + \tilde{u} }{2} \right) *  \left( \cd_x e \right) + \left( \cd[2]_x e \right) * \left( \cd_x \frac{u + \tilde{u} }{2} \right) \right\rangle \\
		={}& \sum_{k=1}^K \left( e_k \right) \left( \cd_x e_k \right) \left( \cd[2]_x \frac{u_k + \tilde{u}_k}{2} \right) \Delta x + \sum_{k=1}^K \left( e_k \right) \left( \cd[2]_x e_k \right) \left( \cd_x \frac{u_k + \tilde{u}_k}{2} \right) \Delta x
	\end{align*}
	and
	\begin{align*}
		&\left\langle e,  \cd_x \left( u * \left( \cd[2]_x u \right) - \tilde{u} * \left( \cd[2]_x \tilde{u} \right) \right) \right\rangle \\
		={}& - \left\langle \cd_x e , \frac{u + \tilde{u}}{2} * \left( \cd[2]_x e \right) + e * \left( \cd[2]_x \frac{u +\tilde{u}}{2} \right) \right\rangle \\
		={}& - \sum_{k=1}^K \left( \cd_x e_k \right) \left( \cd[2]_x e_k \right) \left( \frac{u_k + \tilde{u}_k }{2} \right) \Delta x + \sum_{k=1}^K \left( \cd_x e_k \right) e_k \left( \cd[2]_x \frac{u_k + \tilde{u}_k}{2} \right) \Delta x,
	\end{align*}
	we obtain
	\begin{align*}
	\left\langle \ed{m+1/2}{}, \eta \big( \ud{m}{} \big) - \eta \big( \tud{m}{} \big) \right\rangle 
	&= \sum_{k=1}^K \left( \cd[2]_x e_k \right) \left( e_k \left( \cd_x \frac{u_k + \tilde{u}_k }{2} \right) - \left( \cd_x e_k \right) \left( \frac{u_k + \tilde{u}_k }{2} \right) \right) \Delta x \\
	&= \sum_{k=1}^K \left( \cd[2]_x e_k \right) \left( e_k \left( \cd_x \tilde{u}_k \right) - \left( \cd_x e_k \right) \tilde{u}_k \right) \Delta x.
	\end{align*}
	The upper bound of the first term can be calculated by
	\begin{align*}
		\left| \sum_{k=1}^K \left( \cd[2]_x e_k \right) \left( e_k \left( \cd_x \tilde{u}_k \right) \right) \Delta x \right| 
		&= \left| - \sum_{k=1}^K \left( \fd_x e_k \right) \fd_x \left( e_k \left( \cd_x \tilde{u}_k \right) \right) \Delta x \right| \\
		&=  \left|- \sum_{k=1}^K \left( \fd_x e_k \right) \left( \left( \fd_x e_k \right) \left( \cd_x \tilde{u}_k \right) - e_{k+1} \left( \fd_x \cd_x \tilde{u}_k \right) \right) \Delta x \right| \\
		&\le  \left\| \fd_x e \right\|^2 \sup_{t \in [0,T] } \| u_x (t) \|_{\infty} + \frac{ \left\| \fd_x e \right\|^2 + \left\| e \right\|^2  }{2} \sup_{t \in [0,T]} \| u_{xx} (t) \|_{\infty},
	\end{align*}
	and that of the second term can be calculated by
	\begin{align*}
		\left|- \sum_{k=1}^K \left( \cd[2]_x e_k \right)  \left( \cd_x e_k \right) \tilde{u}_k \Delta x \right|
		&= \left| -\sum_{k=1}^K \left( \frac{1}{2} \bd_x \left( \fd_x e_k \right)^2 \right) \tilde{u}_k \Delta x \right| 
		= \left| \frac{1}{2} \sum_{k=1}^K \left( \fd_x e_k \right)^2 \left( \fd_x \tilde{u}_k \right) \Delta x \right| \\
		&\le \frac{1}{2} \left\| \fd_x e \right\|^2 \sup_{t \in [0,T]} \| u_x (t) \|_{\infty} .
	\end{align*}
	Therefore, by introducing $ D := \max \left\{ \sup_{t \in [0,T] } \| u_x (t) \|_{\infty}, \sup_{t \in [0,T]} \| u_{xx} (t) \|_{\infty} \right\} $,
	we see that
	\[ \fd_t \big\| \fd_x \ed{m}{} \big\|^2 \le \frac{1+D}{2} \big\| \ed{m+1/2}{} \big\|^2 + 2D \big\| \fd_x \ed{m+1/2}{} \big\|^2 + \frac{1}{2} \big\| \taud{m}{} \big\|^2. \]
	Summing up these inequalities for $m$ yields
	\begin{align*}
		\big\| \fd_x \ed{m}{} \big\|^2
		&\le \big\| \fd_x \ed{0}{} \big\|^2 + a \Delta t \big\| \ed{m}{} \big\|_{H^1_K (\Sb)}^2  + 2 a \Delta t \sum_{j=0}^{m-1} \big\| \ed{m}{} \big\|_{H^1_K (\Sb)}^2 
		 + \frac{b}{2} LT ((\Delta x)^2 + (\Delta t)^2 )^2 ,
	\end{align*}
	where $ a = \max \{ (1+D)/2, 2 D \} $.
	Since $ \fd_x \ed{0}{k} = 0 $ holds by the definition of $ \ud{0}{k} $, the inequality proves the lemma.
\end{proof}

Since we have some bound for $ \big\| \fd_x \ed{m}{} \big\| $ that is suitable for the discrete Gronwall lemma,
it is sufficient to obtain a similar bound for $ \big\| \ed{m}{} \big\| $.
However, the standard procedure , which requires the estimation of  $ \fd_t \big\| \ed{m}{} \big\| $, seems to be difficult
because we only have the estimate of the local truncation error with respect to the original form~\eqref{eq_proposed},
which involves the term $ \cd[2]_x \fd_t \ud{m}{k} $.
Therefore, instead of $ \fd_t \big\| \ed{m}{} \big\| $,
we use the estimate of the average error below.

\begin{lemma}\label{lem_ae}
	Let $ \bar{e}^{(m)} $ be the average error defined as $ \bar{e}^{(m)} := (1/L) \sum_{k=1}^K \ed{m}{k} \Delta x $.
	Then,
	\[ \left| \bar{e}^{(m)} \right| \le c ( \Delta x )^2 \]
	holds for some positive constant $c$ depending only on the exact solution $u$.
\end{lemma}

\begin{proof}
	Owing to the preservation of $ \sum_{k=1}^K \ud{m}{k} \Delta x $ and $ \int_{\Sb} u(t,x) \rd x $, we see that
	\begin{align*}
	\left| \bar{e}^{(m)} \right|
	&= \left| \frac{1}{L} \left( \sum_{k=1}^{K} \ud{0}{k} \Delta x - \int_{\Sb} u_0 (x) \rd x \right) + \frac{1}{L} \left( \int_{\Sb} u(m \Delta t,x) \rd x - \sum_{k=1}^K \tud{m}{k} \Delta x \right) \right| \\
	&\le  \frac{1}{L} \left| \sum_{k=1}^{K} u_0 (k \Delta x) \Delta x - \int_{\Sb} u_0 (x) \rd x \right| + \frac{1}{L} \left| \int_{\Sb} u(m \Delta t,x) \rd x - \sum_{k=1}^K \tud{m}{k} \Delta x \right| \\
	&\le  \frac{1}{12} (\Delta x)^2 \| (u_0)_{xx} \|_{\infty} + \frac{1}{12} (\Delta x)^2 \| u_{xx} (m \Delta t) \|_{\infty}.
	\end{align*}
	Therefore, the lemma holds.
\end{proof}

By using the above-mentioned lemmas, we obtain the desired global error estimate as follows.

\begin{theorem}\label{thm_conv}
	Let $u$ be the solution of~\eqref{eq_mHS} satisfying $ u ( t) \in C^7 (\Sb) $ and $ u(\cdot, x) \in C^3 (0,T) $,
	and let $ \ud{m}{k} $ be the numerical solution of the scheme~\eqref{eq_proposed}.
	Then, for any $ m = 1, \dots , M $ and sufficiently small $ \Delta t $ and $ \Delta x $, $ \tud{m}{k} := u( m \Delta t , k \Delta x ) $ satisfies
	\[ \big\| \ud{m}{} - \tud{m}{} \big\|_{H^1_K (\Sb)} \le \sqrt{b'T + c' } \left( (\Delta x)^2 + ( \Delta t)^2 \right) \exp ( a' T ) \]
	for some positive constants $ a' , b' , c' \in \RR $ depending only on $ u $, $ \omega $, and $L$.
\end{theorem}

\begin{proof}
	By using the discrete Poincar\'e--Wirtinger inequality~\eqref{ineq_dPW}, we obtain
	\begin{align*}
		\big\| \ed{m}{} \big\|^2
		&\le \sqrt{L} \big\| \ed{m}{} \big\|_{\infty}^2
		\le L \left(  \big\| \ed{m}{} - \bar{e}^{(m)} \mathbf{1} \big\|_{\infty} + \big| \bar{e}^{(m)} \big| \right)^2 
		\le 2L \left( \big\| \ed{m}{} - \bar{e}^{(m)} \mathbf{1} \big\|_{\infty}^2 + \big| \bar{e}^{(m)} \big|^2 \right) \\
		&\le 2L^2 \big\| \fd_x \ed{m}{} \big\|^2 + 2L c^2 ( \Delta x )^4.
	\end{align*}
	Therefore, we see that
	\begin{align*}
		\big\| \ed{m}{} \big\|_{H^1_K(\Sb)}^2
		&\le (1 + 2 L^2 ) \big\| \fd_x \ed{m}{} \big\|^2 + 2 L c^2 (\Delta x)^4 \\
		&\le \alpha \Delta t \big\| \ed{m}{} \big\|_{H^1_K (\Sb)}^2  + 2 \alpha \Delta t \sum_{j=0}^{m-1} \big\| \ed{j}{} \big\|_{H^1_K (\Sb)}^2 
		 +  L ( \beta T+2c^2) ((\Delta x)^2 + (\Delta t)^2 )^2,
	\end{align*}
	where $ \alpha := (1+2L^2) a $ and $ \beta := (1 + 2 L^2 ) b $.
	Thus, by rewriting this inequality as
	\[ (1 - \alpha \Delta t ) \big\| \ed{m}{} \big\|_{H^1_K(\Sb)}^2
	\le 2 \alpha \Delta t \sum_{j=0}^{m-1} \big\| \ed{j}{} \big\|_{H^1_K (\Sb)}^2  +  L ( \beta T+2c^2) ((\Delta x)^2 + (\Delta t)^2 )^2 \]
	and using the discrete Gronwall lemma (Lemma~\ref{lem_dG}) with the assumption that $ \alpha \Delta t < (1/2)$, we obtain
	\[ \big\| \ed{m}{} \big\|_{H^1_K(\Sb)}^2 \le  2 L ( \beta T + 2 c^2 ) \left( (\Delta x)^2 + (\Delta t)^2 \right)^2 \exp ( 4 \alpha T ), \]
	which proves the theorem.
\end{proof}

The following $L^{\infty}$ error estimate is an immediate corollary of the above-mentioned theorem
owing to the discrete Sobolev lemma (Lemma~\ref{lem_dS}).

\begin{corollary}\label{cor_conv}
	Under the assumption in Theorem~\ref{thm_conv},
	\[  \big\| \ud{m}{} - \tud{m}{} \big\|_\infty \le \hat{L} \sqrt{b'T + c' } \left( (\Delta x)^2 + ( \Delta t)^2 \right) \exp ( a' T )  \]
	holds.
\end{corollary}

\section{Numerical experiment}
\label{sec_ne}

\subsection{Convergence rate}

In this section, we confirm the convergence rate by using the initial condition
\begin{equation}\label{eq_sw}
	u_0 (x) = a \sin ( 2 \pi x ) - ( \pi a)^2,
\end{equation}
where $ a > 0 $ is a parameter.
Here, we set $ L = 1 $ and $ \omega = 1/2 $.
Note that, in~\eqref{eq_sw}, the purpose of the constant term is to ensure that $ \mathcal{F} (u_0) = 0 $.

As the exact solution blows up in finite time for $ a \ge \sqrt{2} / \pi $ (see Proposition~\ref{prop_bu}),
we employ sufficiently small $ a $, i.e., $ a = 1/100 $
(the numerical solution breaks in finite time even when $ a = 1/10 $, as shown in the following section).

To solve the nonlinear equation, we use the fixed point iteration for~\eqref{eq_fpi} and transform the solution $ \vd{m}{k} $ into $ \ud{m}{k} $ by using the relation $ \ud{m}{k} = \ibd_x \vd{m}{k} $.
To illustrate the theory, we employ the step size $ \Delta t $ satisfying $ \Delta t \le \epsilon_1 (2,r) $ (see Remark~\ref{rem_dt}).

We compare the $L^{\infty} $ error of numerical solutions of our scheme at $ T = 10 $
for $ K =  32, 64, 128, 256, 512, 1024$.
Since $ 0.1 \le \epsilon_1 (2, r) $ holds for $K =32$,
we employ $M$ such that $ M = (100/32) K $.
As we do not know the exact solution for this initial condition,
we regard the numerical solution by our scheme~\eqref{eq_proposed} with $ K = 2048 $ and $ M = 6400$ as the reference solution (Fig.~\ref{fig_ref}).
As shown in Fig.~\ref{fig_err}, the $ L^{\infty} $ error decays in the second order, as indicated in Corollary~\ref{cor_conv}.

\begin{figure}[ht]
\pgfplotsset{width=8cm}
\begin{minipage}{0.49\textwidth}
	\centering
\begin{tikzpicture}
\begin{axis}[compat = newest,xlabel={$x$},ylabel={$u$},enlarge x limits=false,ymin=-0.01,ymax=0.012]
\addplot[black,smooth] table {data/mHSsinerefT10.dat};
\end{axis}
\end{tikzpicture}
\caption{Reference solution at $ T = 10$.}
\label{fig_ref}
\end{minipage}
\begin{minipage}{0.49\textwidth}
	\centering
\begin{tikzpicture}
\begin{loglogaxis}[compat = newest,xlabel={$\Delta x$},ylabel={$ L^{\infty} $ Error}]
\addplot[black,mark = x,only marks] table {data/mHSsineErr.dat};
\addplot[smooth] coordinates {
	(0.001,0.000001)
	(0.03,0.0009)
};
\end{loglogaxis}
\end{tikzpicture}
\caption{$ L^{\infty} $ errors of our scheme for each mesh size: `$ \times $' denotes the numerical result, and the straight line represents the slope of the second order.}
\label{fig_err}
\end{minipage}
\end{figure}

\subsection{Blow up phenomena}

In this section, we present a numerical simulation for the initial condition~\eqref{eq_sw} with $ a = 1/10 $.
Here, because the value of $ \| u_{xx} \|_{\infty} $ increases rapidly, we adaptively choose the step size as
$ \Delta t^{(m+1)} = \min \Big\{ \Delta t^{(m)}, \alpha / \big\| \cd[2]_x \ud{m}{} \big\|_{\infty} \Big\} $, where $ \Delta t^{(0)} = 0.0001 \le  \epsilon_1 (2,r) $ and $ \alpha := 1.5 \Delta t^{(0)} \big\| \cd[2]_x \ud{0}{}  \big\|_{\infty} $ ($K = 2048$).
In the numerical experiment, we execute the scheme~\eqref{eq_proposed} with 8,000 steps ($ T \approx 2.054 $).

Figures~\ref{fig_bu_ev} and \ref{fig_bu_ss} show the time evolutions and snapshots of the corresponding numerical solutions.
As shown in these figures, in this time regime, $ \fd_x \ud{m}{k} $ develops a sharp front around $ x \approx 0.33 $, while $ \ud{m}{k} $ is smooth.

\begin{figure}[ht]
	\begin{minipage}{0.49\textwidth}
		\centering
		\pgfplotsset{width=8cm,compat=newest}
		\begin{tikzpicture}
		\begin{axis}[
		xlabel=$x$, 
		ylabel=$t$, 
		zlabel=$u$, zmin = -0.25, zmax=0.02, 
		view={30}{45},
				title=\footnotesize]
		enlarge x limits=false]
		\addplot3[black] table {data/BUSineA10K2048M80kstep1.dat};
		\addplot3[black] table {data/BUSineA10K2048M80kstep2.dat};
		\addplot3[black] table {data/BUSineA10K2048M80kstep3.dat};
		\addplot3[black] table {data/BUSineA10K2048M80kstep4.dat};
		\addplot3[black] table {data/BUSineA10K2048M80kstep5.dat};
		\addplot3[black] table {data/BUSineA10K2048M80kstep6.dat};
		\addplot3[black] table {data/BUSineA10K2048M80kstep7.dat};
		\addplot3[black] table {data/BUSineA10K2048M80kstep8.dat};
		\addplot3[black] table {data/BUSineA10K2048M80kstep9.dat};
		\addplot3[black] table {data/BUSineA10K2048M80kstep10.dat};
		\addplot3[black] table {data/BUSineA10K2048M80kstep11.dat};
		\end{axis}
		\end{tikzpicture}
		
		(a) $ u(t,x) $
	\end{minipage}
\begin{minipage}{0.49\textwidth}
	\centering
	\pgfplotsset{width=8cm,compat=newest}
	\begin{tikzpicture}
	\begin{axis}[
	xlabel=$x$, 
	ylabel=$t$, 
	zlabel=$u_x$, zmin = -3, zmax=1, 
	view={30}{45},
	enlarge x limits=false]
	\addplot3[black] table {data/BUderSineA10K2048M80kstep1.dat};
	\addplot3[black] table {data/BUderSineA10K2048M80kstep2.dat};
	\addplot3[black] table {data/BUderSineA10K2048M80kstep3.dat};
	\addplot3[black] table {data/BUderSineA10K2048M80kstep4.dat};
	\addplot3[black] table {data/BUderSineA10K2048M80kstep5.dat};
	\addplot3[black] table {data/BUderSineA10K2048M80kstep6.dat};
	\addplot3[black] table {data/BUderSineA10K2048M80kstep7.dat};
	\addplot3[black] table {data/BUderSineA10K2048M80kstep8.dat};
	\addplot3[black] table {data/BUderSineA10K2048M80kstep9.dat};
	\addplot3[black] table {data/BUderSineA10K2048M80kstep10.dat};
	\addplot3[black] table {data/BUderSineA10K2048M80kstep11.dat};
	\end{axis}
	\end{tikzpicture}
	
	(b) $ u_x (t,x) $
\end{minipage}
\caption{Time evolution of the corresponding solution for the initial condition~\eqref{eq_sw} with $ a = 1/10 $. The lines are for every 800 steps.}
\label{fig_bu_ev}
\end{figure}
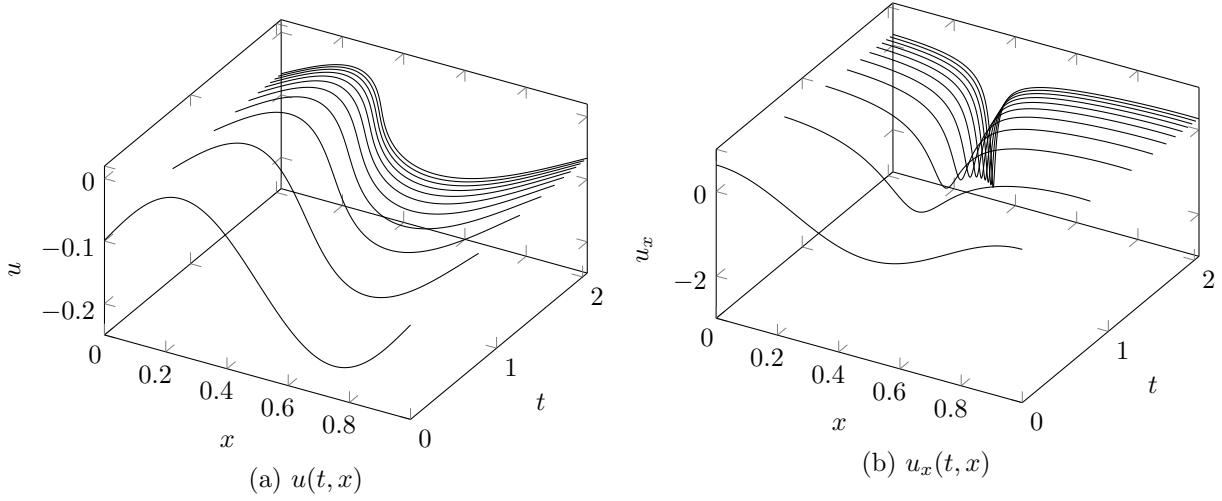

\begin{figure}[ht]
	\begin{minipage}{0.49\textwidth}
	\centering
	\pgfplotsset{width=8cm,compat=newest}
	\begin{tikzpicture}
	\begin{axis}[
	xlabel=$x$, 
	ylabel=$t$, 
	zlabel=$u$, zmin = -0.2, zmax=0, 
	view={0}{0},
	title=\footnotesize]
	enlarge x limits=false]
	\addplot3[black] table {data/BUSineA10K2048M80kstep11.dat};
	\end{axis}
	\end{tikzpicture}
	
	(a) $ u(t,x) $
\end{minipage}
\begin{minipage}{0.49\textwidth}
	\centering
	\pgfplotsset{width=8cm,compat=newest}
	\begin{tikzpicture}
	\begin{axis}[
	xlabel=$x$, 
	ylabel=$t$, 
	zlabel=$u_x$, zmin = -3, zmax=1, 
	view={0}{0},
	enlarge x limits=false]
	\addplot3[black] table {data/BUderSineA10K2048M80kstep11.dat};
	\end{axis}
	\end{tikzpicture}
	
	(b) $ u_x (t,x) $
\end{minipage}
\caption{Snapshots at $ t \approx 2.054 $ of the corresponding solution for the initial condition~\eqref{eq_sw} with $ a = 1/10 $.}
\label{fig_bu_ss}
\end{figure}
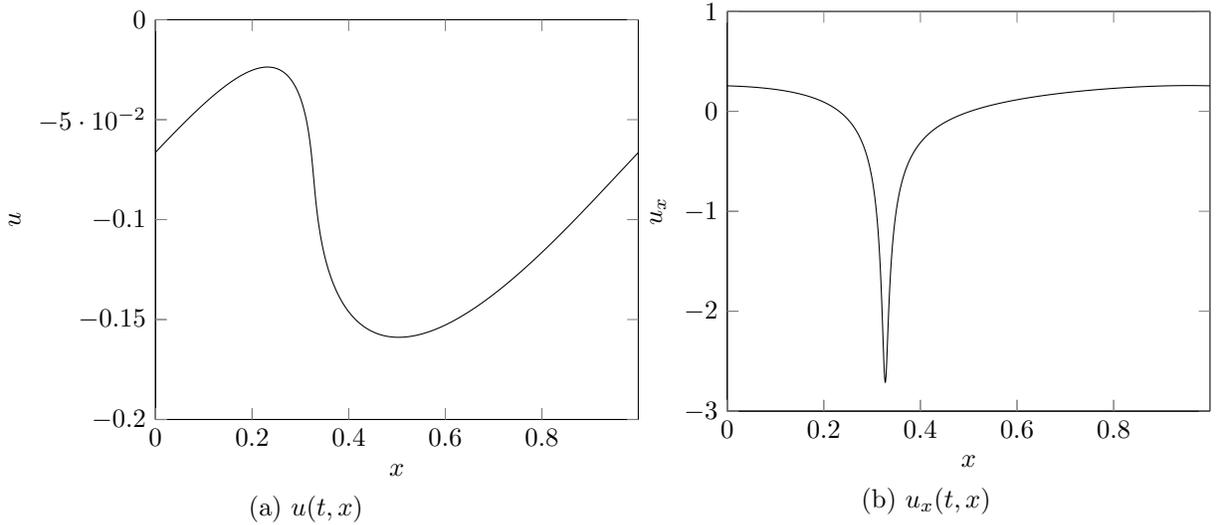

This tendency can be observed more clearly by plotting the time evolution of $ \big\| \fd_x \ud{m}{} \big\|_{\infty} $ and $ \big\| \cd[2]_x \ud{m}{} \big\|_{\infty} $ (see Figure~\ref{fig_bu_norms}).
In other words, at time $ t \approx 2.054 $, the value of $ \big\| \fd_x \ud{m}{} \big\|_{\infty} $ is not large, while that of $ \big\| \cd[2]_x \ud{m}{}  \big\|_{\infty} $ is extremely large.
Moreover, as the slope of $ \big\| \cd[2]_x \ud{m}{} \|_{\infty} $ is larger than that of $ \big\| \fd_x \ud{m}{} \big\|_{\infty} $,
$ \| u_x (t) \|_{\infty} $ is expected to remain finite when $ \| u_{xx} (t) \|_{\infty} $ blows up.
To discuss this point, we show the time evolution of $ \big\| \fd_x \ud{m}{} \big\|_{\infty}^{-1} $ and $ \big\| \cd[2]_x \ud{m}{} \big\|_{\infty}^{-1/2} $ in Figure~\ref{fig_bu_norms_inv}.

\begin{figure}[ht]
\begin{minipage}{0.48\textwidth}
	\centering
	\pgfplotsset{width=8cm,compat=newest}
	\begin{tikzpicture}
	\begin{axis}[compat = newest,xlabel={$ t $},ylabel={$ \big\| \fd_x \ud{m}{} \big\|_{\infty} $},ymin=0,ymax = 4,xmin=0,xmax = 2.4]
	\addplot[black] table {data/BUSineA10K2048M80kux.dat};
	\addplot[dashed,domain = 0:2.3] {( 1.467743302964564-0.536726670065028*x )^(-1)};
	\end{axis}
	\end{tikzpicture}
\end{minipage}
\begin{minipage}{0.48\textwidth}
	\centering
	\pgfplotsset{width=8cm,compat=newest}
	\begin{tikzpicture}
	\begin{axis}[compat = newest,xlabel={$ t $},ylabel={$  \big\| \cd[2]_x \ud{m}{} \big\|_{\infty} $},ymin=0,ymax =200,xmin = 0,xmax = 2.4]
	\addplot[black] table {data/BUSineA10K2048M80kuxx.dat};
	\addplot[dashed,domain = 0:2.3] {( 0.520069041674186 -0.213232180683550 *x )^(-2)};
	\end{axis}
	\end{tikzpicture}
\end{minipage}
\caption{Time evolution of $ \big\| \fd_x \ud{m}{} \big\|_{\infty}  $ (left panel) and $  \big\| \cd[2]_x \ud{m}{} \big\|_{\infty} $ (right panel).
The solid lines represent the numerical solution, and the dashed lines are obtained by linear regression (see Fig.~\ref{fig_bu_norms_inv}).}
\label{fig_bu_norms}
\end{figure}
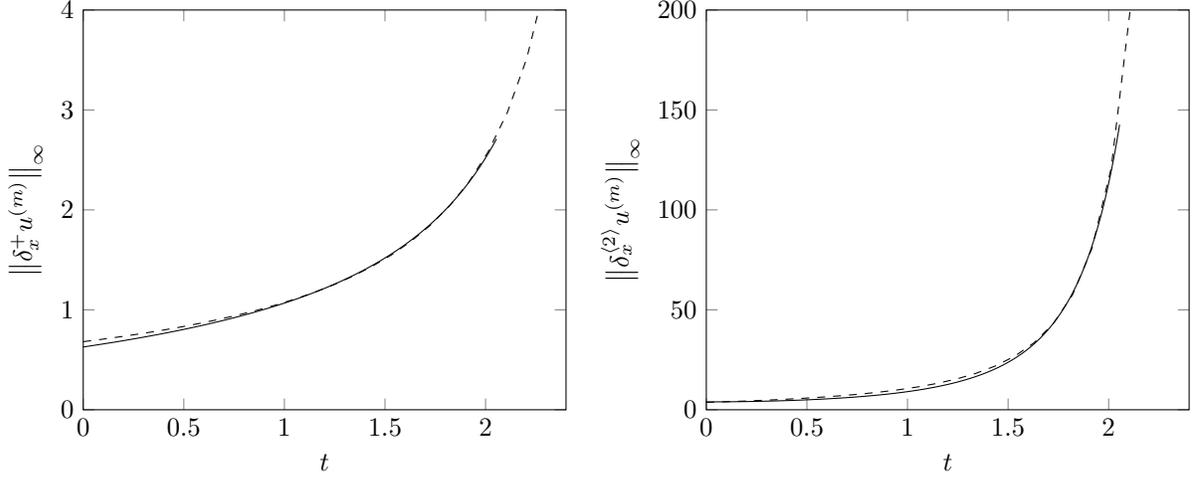

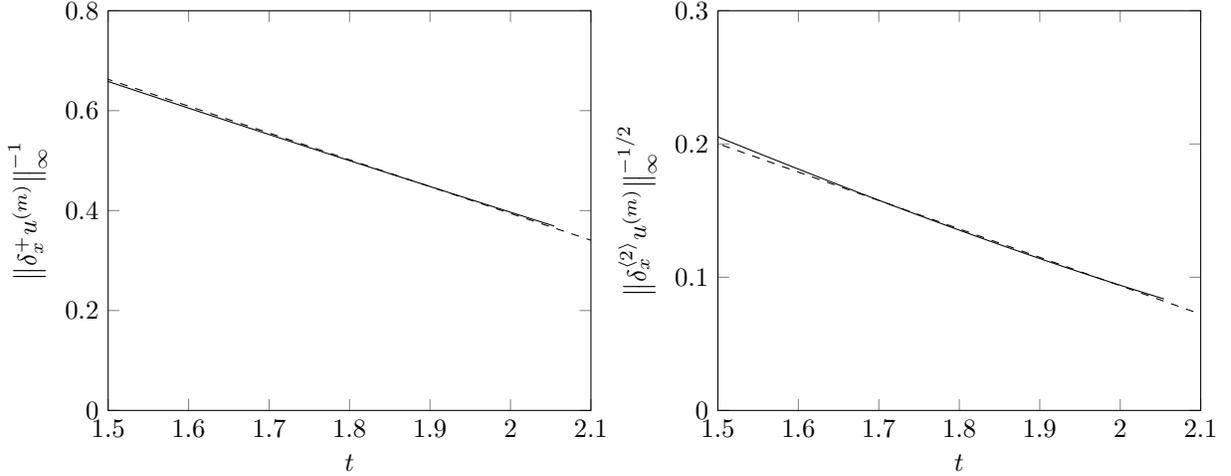
\begin{figure}[ht]
	\begin{minipage}{0.48\textwidth}
		\centering
		\pgfplotsset{width=8cm,compat=newest}
		\begin{tikzpicture}
		\begin{axis}[compat = newest,xlabel={$ t $},ylabel={$ \big\| \fd_x \ud{m}{} \big\|_{\infty} ^{-1} $},ymin=0,ymax = 0.8,xmin =1.5,xmax=2.1]
		\addplot[black] table {data/BUSineA10K2048M80kuxinv.dat};
		\addplot[dashed,domain = 0:2.3] { 1.467743302964564-0.536726670065028*x};
		\end{axis}
		\end{tikzpicture}
	\end{minipage}
	\begin{minipage}{0.48\textwidth}
		\centering
		\pgfplotsset{width=8cm,compat=newest}
		\begin{tikzpicture}
		\begin{axis}[compat = newest,xlabel={$ t $},ylabel={$ \big\| \cd[2]_x \ud{m}{} \big\|_{\infty}^{-1/2} $},ymin=0,ymax =0.3,xmin = 1.5, xmax=2.1]
		\addplot[black] table {data/BUSineA10K2048M80kuxxinv.dat};
		\addplot[dashed,domain = 0:2.3] { 0.520069041674186 -0.213232180683550 *x };
		\end{axis}
		\end{tikzpicture}
	\end{minipage}
\caption{Time evolution of $ \big\| \fd_x \ud{m}{} \big\|_{\infty}^{-1} $ (left panel) and $  \big\| \cd[2]_x \ud{m}{} \big\|_{\infty}^{-1/2} $ (right panel). The solid lines represent the numerical solution, and the dashed lines are obtained by linear regression.}
\label{fig_bu_norms_inv}
\end{figure}

From these numerical results, their blow up rates can be assumed to be
\begin{align*}
\| u_x (t) \|_{\infty} &= \Theta \left( \frac{1}{T^{(2)} -t } \right), &
\| u_{xx} (t) \|_{\infty} &= \Theta \left( \frac{1}{(T^{(\infty)} -t)^2 } \right). &
\end{align*}
Thus, by linear regression for $ \big\| \fd_x \ud{m}{} \big\|_{\infty}^{-1} $ and $ \big\| \cd[2]_x \ud{m}{} \big\|_{\infty}^{-1/2} $ (we use only the last two-thirds of the data for linear regression),
we assume the blow up times as $ T^{(\infty)} \approx 2.4 $ and $ T^{(2)} \approx 2.7 $ (see Table~\ref{tab_blowup}).
From these numerical observations, we expect that the blow up in the sense of Proposition~\ref{prop_bu_uxx} occurs (with $ \| u_x \|_{\infty} < \infty $ at that time).
In other words, there is some initial condition satisfying $ T^{(\infty)} < T^{(2)} $, i.e.,
$ \| u_{xx} \|_{\infty} $ blows up first and $ \| u_x \|_{\infty} $ blows up subsequently.

\begin{table}[ht]
	\centering
	\caption{Estimated blow up times $ T^{(2)} $ and $ T^{(\infty)} $ by linear regression for each $K$.
	$ R^2 $ denotes the coefficient of determination for each linear regression.}
	\label{tab_blowup}
	\begin{tabular}{c||c|c||c|c}
		\hline
		 & \multicolumn{2}{c||}{$ \| u_x \|_{\infty} $} & \multicolumn{2}{c}{$ \| u_{xx} \|_{\infty} $} \\ \hline
		 $K$ & $ T^{(2)} $ & $ R^2$ & $ T^{(\infty)} $ & $ R^2 $ \\ \hline \hline
		 128	& 3.0467 & 0.99274 & 2.5323 & 0.98572 \\
		 	\hline
		 256	& 2.8509 & 0.99865 & 2.4308 & 0.99782 \\
		 	\hline
		 512	& 2.7861 & 0.99975 & 2.4450 & 0.99953 \\
		 	\hline
		 1024	& 2.7680 & 0.99995 & 2.4497 & 0.99848 \\
		 	\hline
		 2048	& 2.7648 & 0.99997 & 2.4390 & 0.99883 \\
		 	\hline
	\end{tabular}
\end{table}

This observation is just numerical evidence of the blow up phenomena.
However, we believe that this observation can encourage further mathematical analysis of the mHS equation.

\section{Concluding remarks}
\label{sec_cr}

In this paper, we derived a stable and convergent finite difference method for the initial value problem for the mHS equation~\eqref{eq_mHS} on the periodic domain.
To the best of our knowledge, our contribution is the first attempt to rigorously justify the numerical method for evolutionary equations with a mixed derivative
except for the simplest case~\cite{CRR2017}.
Hence, it could open the door to numerical analysis of such equations.

Moreover, by using our numerical scheme,
we observed some interesting blow up phenomena,
i.e., the blow up times of $ \| u_{xx} (t) \|_{\infty} $ and $ \| u_x (t) \|_{\infty} $ are different.
Some rigorous justification of these phenomena is left to future work.

\section*{Acknowledgment}
	The author is grateful to Yuto Miyatake and Takayasu Matsuo for their insightful comments.

%
%

\appendix

\section{Relation to the method of Miyatake, Cohen, Furihata, and Matsuo~\cite{MCFM2017}}
\label{sec_relation}

Miyatake, Cohen, Furihata, and Matsuo~\cite{MCFM2017} derived the initial value problem
\begin{equation}\label{eq_mHS_MCFM}
\begin{cases}
u_{t} = \left( \partial_x^{\dagger} \right)^2 \left( -\frac{1}{2} \left(u^2 \right)_{xxx} + 2 \omega u_x + \frac{1}{2} \left( u_x^2 \right)_x \right) \qquad & (t \in (0,T), x \in \Sb ), \\
u (0,x) = u_0 (x) & (x \in \Sb),
\end{cases}
\end{equation}
by assuming that the underdetermined form~\eqref{eq_mHS2} has a traveling wave solution ($ \partial_x^{\dagger} $ is a pseudo-inverse of the differential operator; see \cite{MCFM2017} for details).
Accordingly, they devised underdetermined discretization~\eqref{eq_dvdm_mHS2_con} for preservation of $ \mathcal{H} $
and employed the Moore--Penrose inverse $ \icd[2]_x $ of the second-order central difference $ \cd[2]_x $.
Thus, their scheme can be written in the form
\begin{equation}\label{eq_MCFM}
\begin{cases}
\fd_t \ud{m}{k} = \icd[2]_x \left( \mathcal{A}_{\rd} \big( \ud{m+1/2}{} \big) \ud{m+1/2}{k} \right) \quad & (m=0,1,\dots,M-1;k \in \mathbb{Z}),\\
\ud{0}{k} = u_0 (k \Delta x) & (k \in \mathbb{Z}),\\
\ud{m}{k+K} = \ud{m}{k} & ( m = 0,1,\dots,M-1, k \in \mathbb{Z} ).
\end{cases}
\end{equation}
Although they were interested in traveling wave solutions, this scheme can be employed for the problem~\eqref{eq_mHS_MCFM} even when the solution is not actually a traveling wave.

Actually, initial value problems~\eqref{eq_mHS_ode} and~\eqref{eq_mHS_MCFM} are equivalent for sufficiently smooth solutions.
Thus, both our scheme~\eqref{eq_proposed} and that of Miyatake, Cohen, Furihata, and Matsuo~\cite{MCFM2017}~\eqref{eq_MCFM} can be regarded as a numerical scheme for the artificial initial value problem~\eqref{eq_mHS_ode}.
Moreover,
although the derivations of these schemes are quite different,
these methods happen to be equivalent for the generalized problem~\eqref{eq_mHS_ode} (or \eqref{eq_mHS_MCFM}) as follows.

\begin{theorem}\label{thm_equiv}
	If $ \ud{m}{k} $ is a solution of \eqref{eq_proposed}, then it is also a solution of~\eqref{eq_MCFM}, and vice versa.
\end{theorem}

\begin{proof}
	To prove the former part, we assume that $ \ud{m}{k} $ is a solution of~\eqref{eq_proposed}.
	Then, owing to the conservation of $ \sum_{k=1}^K \ud{m}{k} \Delta x $,
	we observe that $ \fd_t \ud{m}{} \in \Car \big( \cd[2]_x \big) $ holds.
	Therefore, on the basis of the property of the Moore--Penrose inverse, it holds that
	\[ \fd_t \ud{m}{k} = \icd[2]_x \left( \mathcal{A}_{\rd} \big( \ud{m+1/2}{} \big) \ud{m+1/2}{k} \right), \]
	which proves the former part.
	
	On the other hand, for the latter part, we assume that $ \ud{m}{k} $ is a solution of~\eqref{eq_MCFM}.
	By operating with $ \cd[2]_x $ on both sides of the scheme in \eqref{eq_MCFM}, we see that
	\[ \cd[2]_x \fd_t \ud{m}{k} = \cd[2]_x \icd[2]_x \left( \mathcal{A}_{\rd} \big( \ud{m+1/2}{} \big) \ud{m+1/2}{k} \right). \]
	Since $ \cd[2]_x \icd[2]_x = P $, it is sufficient to prove that $ \mathcal{A}_{\rd} \big( \ud{m+1/2}{} \big) \ud{m+1/2}{k} \in \Range \big ( \cd[2]_x \big) $.
	Therefore,
	\begin{align*}
	\sum_{k=1}^K \mathcal{A}_{\rd} \left( \ud{m+1/2}{} \right) \ud{m+1/2}{k} \Delta x
	&= \sum_{k=1}^K \cd_x \left( \left( 2 \omega - \left( \cd[2]_x \ud{m+1/2}{k} \right) \right) \ud{m+1/2}{k}  \right) \Delta x \\
	& \qquad \qquad - \sum_{k=1}^K \left( \cd[2]_x \ud{m+1/2}{k} \right) \left( \cd_x \ud{m+1/2}{k} \right) \Delta x \\
	&= - \sum_{k=1}^K \frac{1}{2} \fd_x \left( \bd_x \ud{m+1/2}{k} \right)^2 \Delta x = 0
	\end{align*}
	proves the latter part.
\end{proof}

\end{document}